\documentclass[a4paper,twoside,11pt,reqno,francais]{smfart}
\usepackage{amssymb,url,xspace,smfthm,smfenum}

\usepackage[T1]{fontenc}
\usepackage[english,francais]{babel}
\usepackage{hyperref}

\makeatletter
    \def\ps@copyright{\ps@empty
    \def\@oddfoot{\hfil\small\copyright 1997, \SMF}}
\makeatother

\newcommand{\SMF}{Soci\'et\'e Ma\-th\'e\-Ma\-ti\-que de France}
\newcommand{\BibTeX}{{\scshape Bib}\kern-.08em\TeX}
\newcommand{\T}{\S\kern .15em\relax }
\newcommand{\AMS}{$\mathcal{A}$\kern-.1667em\lower.5ex\hbox
        {$\mathcal{M}$}\kern-.125em$\mathcal{S}$}

\usepackage{multicol}
\usepackage{geometry}
\usepackage{idxlayout}
\usepackage{t1enc}	
\usepackage[utf8]{inputenc}	
\usepackage[T1]{fontenc}

\usepackage{amsmath}
\usepackage{amsfonts}
\usepackage{amssymb} 
\usepackage{color,euscript,graphicx,tabularx,enumerate}
\usepackage{graphicx}
\usepackage{version}

\makeindex
\def\h1{\hspace{1cm}}
\def\h2{\hspace{2cm}}\def\h3{\hspace{3cm}}
\def\h4{\hspace{4cm}}\def\h5{\hspace{5cm}}
 \definecolor{grey}{rgb}{0.75,0.75,0.75}
\definecolor{orange}{rgb}{1.0,0.5,0.5}
\definecolor{brown}{rgb}{0.5,0.25,0.0} 
\definecolor{pink}{rgb}{1.0,0.5,0.5}

\def\dis{\displaystyle}
\newcommand{\ec}[1]{\vskip 3mm\noindent\setlength{\fboxrule}{0.5pt}
\textcolor{black\begin{minipage}[b]{15cm}
  }{\fbox{
\parbox[c][\height]{5cm}{\textcolor{black}{ #1 \vskip 1mm\noindent }}}}} 
\def\paragraph#1{\textit{#1}.}
\def\C{{\mathbb C}}

\newtheorem{tm}{Th\'eor\`eme}
\newtheorem{pp}{Proposition}

\newtheorem{lm}{Lemme}
\newtheorem{cl}{Corollaire}
\newtheorem{defin}{D\'efinition}
\newtheorem{req}{Remarque}

\newtheorem{avt}{Avertissement}
\newcommand{\cRM}[1]{\MakeUppercase{\romannumeral #1}}  

\newcommand{\faitdisparaitre}[1]{}

 \def\cqfd{\unskip\kern 6pt\penalty 500
\raise -2pt\hbox{\vrule\vbox to5pt{\hrule width 4pt
\vfill\hrule}\vrule}}

\def\Ker{\textsf{\,Ker\,}}

 \def\dfl{\textsf{\,dfl\,}}

\def\A{\mathbf{A}}
\def\A2{{\mathbf{A}}^2}

\def\AB0{{\A}^2_{B( {0} ,\,{1})}}
\def\AB#1#2{{\A}^2_{B( {#1} ,\,{#2})}}

\def\D{\mathbf{D}}

\def\N{\mathbf{N}}

\def\C{\mathbf{C}}

\def\E{\mathbf{E}}
\def\F{\mathbf{F}}

\def\O{\mathbf{O}}
\def\T{\mathbf{T}}
\def\S{\mathbf{S}}
\def\L{\mathbf{L}}

\def\mE{\mathcal{E}}

\def\binomial#1#2{{{#1}\choose{#2}}}

\def\minus{\backslash}
\def\min{\mathrm{min}}

\def\monomials#1{\mathcal{M}}

\def\rank{\mathrm{rang}}

\def\eval{\, eval\, }
\def\schur{\, Schur\, }
\def\vect{\, vec\,}

\def\dis{\displaystyle}

\def\scr{\scriptstyle}

\title[Approximation num\'erique de racines isol\'ees multiples]{Approximation num\'erique de racines isol\'ees multiples de syst\`emes analytiques}
\date{Version du \today}
\author{M.~Giusti}
\address{Marc Giusti \\Laboratoire LIX  \\ Campus de l'\'Ecole
  Polytechnique \\
1 rue Honor\'e d'Estienne d'Orves\\B\^atiment Alan
Turing\\CS35003\\91120 Palaiseau\\ France.} 
\email{Marc.Giusti@Polytechnique.fr} 
 \author{J.-C.~Yakoubsohn} 
\address{Jean-Claude Yakoubsohn \\ Institut de Math\'ematiques de
  Toulouse \\  Universit\'e Paul Sabatier \\ 118 route de Narbonne
  \\ 31062 Toulouse Cedex 9\\
France.}
\email{yak@mip.ups-tlse.fr} 
\altkeywords{systems of equations, singular roots, deflation, numerical rank, 	evaluation}
\begin{document}
\begin{abstract}
 L'approximation d'une racine isol\'ee multiple est un probl\`eme
    difficile. En effet la racine peut m\^eme \^etre r\'epulsive
  pour une m\'ethode de point fixe comme la m\'ethode de Newton. La
  litt\'erature sur le sujet est vaste mais les r\'eponses propos\'ees
  pour r\'esoudre ce probl\`eme ne sont pas satisfaisantes. Des
  m\'ethodes num\'eriques qui permettent de faire une analyse locale
  de convergence sont souvent \'elabor\'ees sous des hypoth\`eses
  particuli\`eres. Ce
  point de vue privil\'egiant l'analyse num\'erique n\'eglige la
  g\'eom\'etrie et la structure de l'alg\`ebre locale. C'est
    ainsi qu'ont \'emerg\'e des m\'ethodes qualifi\'es de
  symboliques-num\'eriques. Mais l'analyse num\'erique pr\'ecise de
  ces m\'ethodes pourtant riches d'enseignement n'a pas
  \'et\'e faite. 

  Nous proposons dans cet article une m\'ethode de type
  symbolique-num\'erique dont le traitement num\'erique est
  certifi\'e.  L'id\'ee g\'en\'erale est de construire une suite finie
  de syst\`emes admettant la m\^eme racine, appel\'ee \emph{suite de d\'eflation}, telle que la
  mutiplicit\'e de la racine chute strictement entre deux syst\`emes
  successifs. La racine devient ainsi r\'eguli\`ere lors du dernier
  syst\`eme. Il suffit alors d'en extraire un syst\`eme carr\'e
  r\'egulier pour obtenir ce que nous appelons \emph{syst\`eme
    d\'eflat\'e}.  Nous avions d\'ej\`a d\'ecrit la construction
  de cette suite de d\'eflation quand la racine est connue.
  L'originalit\'e de cette \'etude consiste d'une part \`a d\'efinir
  une suite de d\'eflation \`a partir d'un point proche de la racine
  et d'autre part \`a donner une analyse num\'erique de cette
  m\'ethode.  Le cadre fonctionnel de cette analyse est celui des
  syst\`emes analytiques constitu\'es de fonctions de carr\'e
  int\'egrable. En utilisant le noyau Bergman, noyau reproduisant de
  cet espace fonctionnel, nous donnons une $\alpha$-\emph{th\'eorie
    \`a la Smale} de cette suite de d\'eflation.  De plus nous
  pr\'esentons des r\'esultats nouveaux relatifs \`a la
  d\'etermination du rang num\'erique d'une matrice et \`a celle de la
  proximit\'e \`a z\'ero de l'application \'evaluation.  Comme
  cons\'equence importante nous donnons un algorithme de calcul d'une
  suite de d\'eflation qui est \emph{libre de $\varepsilon$},
  quantit\'e-seuil qui mesure l'approximation num\'erique, dans le
    sens que les entr\'ees de cet algorithme ne comportent pas
  la variable $\varepsilon$. 
\end{abstract}
\begin{altabstract}
The approximation of a multiple isolated root is a difficult
    problem. In fact the root can even be a repulsive root for a fixed
    point method like the Newton method. However there exists a huge
    literature on this topic but the answers given are not
    satisfactory. Numerical methods allowing a local convergence
    analysis work often under specific hypotheses. This viewpoint
    favouring numerical analysis forgets the geometry and the
    structure of the local algebra. Thus appeared so-called
    symbolic-numeric methods, yet full of lessons, but their precise
    numerical analysis is still missing.\\
    We propose in this paper a method of symbolic-numeric kind, whose
    numerical treatment is certified. The general idea is to construct
    a finite sequence of systems, admitting the same root, and called
    the \emph{deflation sequence}, so that the multiplicity of the
    root drops strictly between two successive systems. So the root
    becomes regular. Then we can extract a regular square system we
    call \emph{deflated system}. We described already the
    construction of this deflated sequence when the singular root is
    known. The originality of this paper consists on one hand to
    construct a deflation sequence from a point close to the root and
    on the other hand to give a numerical analysis of this
    method. Analytic square integrable functions build the functional
    frame. Using the Bergman kernel, reproducing kernel of this
    functional frame, we are able to give a $\alpha$-\emph{theory}
    \emph{\`a la Smale}. Furthermore we present new results on the
    determinacy of the numerical rank of a matrix and the closeness to
    zero of the evaluation map. As an important consequence we give an
    algorithm computing a deflation sequence \emph{free of
      $\varepsilon$}, threshold quantity measuring the numerical
    approximation, meaning that the entry of this algorithm does
    not involve the variable $\varepsilon$.
\end{altabstract}
\altkeywords{systems of equations, singular roots, deflation, numerical rank, 	evaluation}
\maketitle 
\def\smfbyname{}
\tableofcontents
\vspace{0.cm}

\def\indexspace{}
\begin{minipage}[b]{24cm}
  \begin{multicols}{3}
   \begin{minipage}{7cm}
\vspace{2cm}

\begin{theindex}

\item $\zeta$\dotfill 4
\item $\mu(\zeta)$\dotfill 4
\item $f$\dotfill 4
   \item $I$\dotfill 4
\item $\C\{x-\zeta\}$\dotfill 4
\item $I\C\{x-\zeta\}$\dotfill 4
\item $a_k(M)$\dotfill 9
  \item $b_k(M)$\dotfill 9
  \item $g_k(M)$\dotfill 9
  \item $s(\lambda)$\dotfill 9
  \item $s_k$\dotfill 9
 \item $\varepsilon$-rang\dotfill 10
\item $M_\varepsilon$\dotfill 10
\item $p(\lambda)$\dotfill 11
  \item $q(\lambda)$\dotfill 11
 \item $t$\dotfill 11
  \item $\omega$\dotfill 13
\item $\A2(\omega,R_{\omega})$\dotfill 13
\item $R_\omega$\dotfill 13
  \item $||f||$\dotfill 14
 \item $\nu_x$\dotfill 15
  \item $H(z,x)$\dotfill 15
  \item $\alpha_0$\dotfill 17
   \item $c_0$\dotfill 17
  \item $eval_x$\dotfill 17
  \item $\varepsilon$-valuation\dotfill 21
   \item $\vect$\dotfill 22
   \item $\Delta_k$\dotfill 22
   \item $S(f)$\dotfill 22
\item $\schur(M)$\dotfill 21
  \item $K(f)$\dotfill 23
   \item $\dfl(f)$\dotfill 23
  \item $\ell$\dotfill 23
  \item $N_{\textsf  {\tmspace  +\thinmuskip {.1667em}dfl\tmspace  +\thinmuskip {.1667em}}(f)}$\dotfill 		23
\item $\bar{\gamma}(f, \zeta)$\dotfill 27
  \item $\alpha(f, x)$\dotfill 28
  \item $\beta(f, x)$\dotfill 28
  \item $\gamma(f, x)$\dotfill 28
  \item $\lambda(f, x)$\dotfill 28
  \item $\mu(f, x)$\dotfill 28
\item $\theta$\dotfill 28
\item $[F]_\zeta$\dotfill 32
\end{theindex}
\end{minipage}
  \end{multicols}
\end{minipage}

\pagebreak

\section{ $\,$Syst\`emes \'equivalents et multiplicit\'e}
L'article \textbf{\textsf{Multiplicity hunting and approximating
    multiple roots of polynomial systems}}~\cite{GY2} fut \'ecrit par
les deux auteurs de mani\`ere purement heuristique. Nous en
pr\'esentons ici une analyse num\'erique, en simplifiant au passage
l'algorithme initialement exhib\'e.

\begin{defin}
  Une solution $\zeta$ \index{$\zeta$} d'un syst\`eme analytique $f=0$
  \index{$f$} (d\'efini dans un voisinage de $\zeta$) est dite
  \emph{isol\'ee} et \emph{multiple} si~:
\begin{itemize}
\item[1--] il existe un voisinage de $\zeta$ o\`u $\zeta$ est la seule
  solution de $f=0$~;
 \item[2--] la matrice Jacobienne $Df(\zeta)$ n'est pas de rang plein.
\end{itemize}
\end{defin}

Nous utiliserons indiff\'eremment les mots multiple et
singulier. En particulier nous appellerons un syst\`eme
\emph{singulier} s'il admet une solution singuli\`ere. De m\^eme
nous emploierons indistinctement les vocables \emph{solution} ou
\emph{racine} d'un syst\`eme. Afin de simplifier la lecture, la notation $f$ d\'esignera indistinctement une seule \'equation ou un syst\`eme d'\'equations.
Implicitement les boules seront toujours ouvertes.
\\
Remarquons que la premi\`ere hypoth\`ese implique que le nombre
d'\'equations $s$ est plus grand ou \'egal au nombre $n$ de variables. Notons
aussi que ce cadre inclut le cas d'un syst\`eme analytique obtenu par
localisation d'un syst\`eme alg\'ebrique.\\

Nous avions expliqu\'e dans \cite{GY2} comment d\'eriver un syst\`eme
\emph{r\'egulier} (c'est-\`a-dire admettant $\zeta$ comme solution
avec la matrice jacobienne $Df(\zeta)$ de rang plein) d'un
syst\`eme singulier~: \'evidemment ceci sous l'hypoth\`ese que la
solution $\zeta$ est exactement connue. Nous avions alors formalis\'e cette
transformation par la notion de syst\`emes \emph{\'equivalents}  en un
point $\zeta$.\\

Notons que cette transformation est obtenue \textbf{\textsf{sans ajout
    de nouvelles variables}} (trait important que nous soulignons).\\

 
La \emph{multiplicit\'e} d'une racine est un important invariant.
Dans le cas o\`u il n'y a qu'une variable et qu'une \'equation, la
multiplicit\'e d'une racine est exactement le nombre de d\'eriv\'ees
qui s'annulent en la racine, propri\'et\'e qui malheureusement n'est
plus valable dans le cas g\'en\'eral. Il faut alors introduire une
machinerie bien plus compliqu\'ee.\\

Appelons~:
\begin{itemize}
\item[1--]
$\C\{x-\zeta\}$ l'alg\`ebre des germes de fonctions analytiques en
$\zeta$, c'est-\`a-dire l'anneau local des s\'eries convergentes dans un
voisinage de $\zeta$, d'id\'eal maximal engendr\'e par
$x_1-\zeta_1,\ldots, x_n-\zeta_n$~; 
\item[2--] $I\C\{x-\zeta\}$ l'id\'eal induit engendr\'e par l'id\'eal
  $I = I(f):=<f_1,\ldots, f_s>$.
\end{itemize}

\begin{defin}
  La multiplicit\'e $\mu(\zeta)$\index{$\mu(\zeta)$} d'une racine
  isol\'ee $\zeta$ est d\'efinie comme la dimension de l'espace quotient
   $\C\{x-\zeta\}/I\C\{x-\zeta\}$.
\end{defin}

Relativement \`a un ordre local $<$ compatible de $\C\{x-\zeta\}$,
nous notons $LT_<(I\C(x-\zeta))$ l'id\'eal engendr\'e par les termes
dominants de tous les \'el\'ements de $I\C\{x-\zeta\}$.
\begin{defin}
Une base standard (minimale) de $I\C\{x-\zeta\}$ est un ensemble
(fini) de s\'eries de $I\C\{x-\zeta\}$ dont les termes dominants
engendrent minimalement $LT_<(I\C(x-\zeta))$.
\end{defin}
Il existe alors un nombre fini de mon\^omes, appel\'es mon\^omes
standard, qui n'appartiennent pas \`a $I$. Le th\'eor\`eme suivant
est classique dans la litt\'erature sur les bases standard, voir~\cite{cox05} page 178. 
\begin{tm}\label{coxp178}
Les assertions suivantes sont \'equivalentes~:
 \begin{itemize}
 \item[1--] La racine $\zeta$ est isol\'ee~;
 \item[2--] 
$dim\, \C \{x-\zeta\}/I\C(x-\zeta)$ est fini~;
 \item[3--] $dim\, \C \{x-\zeta\}/LT_<(I\C(x-\zeta))$ est fini~;
 \item[4--] Il y a seulement un nombre fini de mon\^omes standard~;
 \end{itemize}
Qui plus est, quand n'importe laquelle de ces conditions est
satisfaite, nous avons~:
 $$\mu(\zeta)=dim\, \C
 \{x-\zeta\}/LT_{<}(I\C(x-\zeta))= \textrm{nombre de 
   mon\^omes standard.}$$
\end{tm}
Dans le cas particulier d'un syst\`eme polynomial localis\'e, dont
quelque entier $d$ borne sup\'erieurement le degr\'e total des 
\'equations, $d^n$ constitue une  borne sup\'erieure pour la
multiplicit\'e de toute racine multiple.
\section{ $\,$Ce que contient cette \'etude}\label{sec2}
Approcher une racine isol\'ee multiple est difficile car la racine
peut \^etre r\'epulsive pour une m\'ethode de point fixe comme la
m\'ethode de Newton (voir l'exemple donn\'e par Griewank et
Osborne~\cite{GO83}, p. 752). Dans ce cas de racine isol\'ee, l'arsenal
des techniques d\'evelopp\'ees si le rang de la matrice jacobienne est
de rang constant ne s'applique pas.  Le cas de rang constant, qui
comprend respectivement les cas surjectif et injectif, est bien
analys\'e dans l'ouvrage de J.-P. Dedieu~\cite{dedieu06}. Il existe une
vaste litt\'erature sur ce sujet, voir par exemple les articles
\cite{xuli2008},~\cite{ArHi2011} et les
r\'ef\'erences qu'ils contiennent.\\

Afin de pallier cet inconv\'enient, nous construisons une suite finie
de syst\`emes \'equivalents, appel\'ee \emph{suite de d\'eflation},
telle que la mutiplicit\'e de la racine chute strictement entre deux
syst\`emes successifs. La racine devient ainsi r\'eguli\`ere lors du
dernier syst\`eme. Il suffit alors d'en extraire un syst\`eme carr\'e
r\'egulier pour obtenir ce que nous appelons \emph{syst\`eme
  d\'eflat\'e}. L'op\'erateur de Newton \emph{singulier} est juste
l'op\'erateur de Newton classique associ\'e au syst\`eme
d\'eflat\'e.\\

Comment construisons-nous cette suite de d\'eflation~? Commen\c cons par expliquer l'idée g\'en\'erale de cette construction quand la racine $\zeta$ est connue. Tout d'abord on remplace les équations par leur gradients tant que ceux-ci s'annulent en $\zeta$. On obtient ainsi un syst\`eme \'equivalent au syst\`eme initial avec la propri\'et\'e suivante : chaque ligne de sa matrice jacobienne
n'est pas identiquement nulle en $\zeta$. On appelle \emph{s\'election} cette op\'eration qui consiste \`a remplacer les \'equations par leurs d\'eriv\'ees \`a un certain ordre. Ensuite,
tant que la
matrice jacobienne n'est pas de rang plein en la racine, c'est qu'il
existe des relations entre les lignes (respectivement entre les
colonnes), qui sont donn\'ees par un compl\'ement de Schur.  Nous
appelons \emph{d\'enoyautage} l'op\'eration qui consiste \`a ajouter
\`a certaines \'equations initiales les \'el\'ements de ce compl\'ement
de Schur. Nous montrerons en section~\ref{Mult-drops-kerneling} qu'apr\`es les op\'erations de s\'election et de d\'enoyautage effectu\'ees sur un syst\`eme  singulier, nous obtenons un syst\`eme \'equivalent o\`u la multiplicit\'e
de la racine  a strictement chut\'e. La m\'ethode de d\'eflation consiste \`a  it\'erer cette
construction, c'est \`a dire \`a faire suivre une op\'eration de s\'election par une op\'eration de 
d\'enoyautage. Le nombre d'it\'erations n\'ecessaires pour obtenir un
syst\`eme r\'egulier est \emph{l'\'epaisseur} de la suite de d\'eflation. Un syt\`eme d\'eflat\'e du syst\`eme initial est un syst\`eme carr\'e extrait du syst\`eme 
r\'egulier obtenu par la m\'ethode de d\'eflation.
L'id\'ee de la m\'ethode est en fait assez naturelle. Pour s'en persuader, il suffit de se reporter \`a l'exemple 
illustr\'e en sous-section~\ref{exemple_exact}.
Pour conclure ce bref aper\c cu de la m\'ethode de d\'eflation, il faut souligner tout d'abord qu'il n'y a pas unicit\'e de la suite de d\'eflation. Ensuite que l'op\'eration de s\'election est en fait une succession d'op\'erations de d\'enoyautage dans le cas de nullit\'e du rang de la matrice jacobienne. Ceci sera d\'evelopp\'e en section~\ref{Kern-Sing-Newton}.

L'originalit\'e de cette \'etude consiste d'une part \`a d\'efinir 
 une m\'ethode de d\'eflation \`a partir d'un point $x_0$ proche de la racine $\zeta$ et d'autre part \`a donner une analyse num\'erique de cette m\'ethode.
Ceci constituera les sections~\ref{Kern-Sing-Newton} et~\ref{Mult-drops-kerneling}. De plus nous estimerons le rayon d'une boule centr\'ee  en $\zeta$
dans laquelle  le rang num\'erique de $Df(x_0)$ est \'egal au rang de $Df(\zeta)$ pour tout point $x_0$ de cette boule.

Le but \'etant d'effectuer l'analyse de cette suite de d\'eflation,
cette \'etude se place dans le contexte  des syst\`emes analytiques
constitu\'es de fonctions de carr\'e int\'egrable. Ainsi nous pouvons
repr\'esenter une fonction et ses d\'eriv\'ees par un noyau
reproduisant efficace~: le noyau  de Bergman. Ce cadre fonctionnel est
d\'ecrit en section~\ref{functional}.

Qui plus est, notre \'etude est \emph{libre de $\varepsilon$}
(quantit\'e-seuil qui mesure l'approximation num\'erique) dans le sens
suivant~:
\begin{defin}
  Un algorithme num\'erique est dit libre de $\varepsilon$ si les
  entr\'ees de cet algorithme ne comportent pas la variable
  $\varepsilon$.
\end{defin}
La construction d'une suite de d\'eflation pr\'esent\'ee dans la table~\ref{dfl_table} est libre de $\varepsilon$ sous l'hypoth\`ese que la
norme d\'efinie dans la section~\ref{functional} (ou \`a tout le
moins une borne sup\'erieure) soit donn\'ee. Pour cela nous
pr\'esentons des r\'esultats nouveaux afin de d\'eterminer via des
algorithmes libres de $\varepsilon$:
\begin{itemize} 
\item[1--] le rang num\'erique d'une matrice, dans la section~\ref{rank_sec}~;
\item[2--] la proximit\'e \`a z\'ero de l'application \'evaluation,
  voir la section~\ref{evaluation_sec}.
\end{itemize} 

Nous terminons cette \'etude par l'$\alpha$-th\'eorie de Smale de cet
op\'erateur de Newton singulier en utilisant l'analyse induite par le noyau
de Bergman.
Nous commen\c cons par donner en  section~\ref{sec_rouche_regulier} un $\alpha$-th\'eor\`eme, c'est \`a dire une condition de
l'existence d'une racine  obtenue gr\^ace au th\'eor\`eme de 
Rouch\'e. Toujours dans le cas r\'egulier, nous  \'etablissons ensuite  en section 
\ref{sec_NS_gamma}, un
$\gamma$-th\'eor\`eme, c'est-\`a-dire un r\'esultat qui exhibe le rayon d'une boule de
convergence quadratique pour l'op\'erateur de Newton. 
L'analyse du cas singulier est l'application au syst\`eme d\'eflat\'e des  $\alpha$-th\'eor\`eme et $\gamma$-th\'eor\`eme respectivement obtenus
dans le cas r\'egulier. Ceci est r\'ealis\'e en section~\ref{sec_alpha_gamma}.
Les r\'esultats de cette section  d\'ependent de l'estimation de la
quantit\'e $\gamma$ du syst\`eme d\'eflat\'e. En notant par $\gamma_0$ (respectivement, $\gamma_\ell$) la quantit\'e $\gamma$ en la racine du syst\`eme initial (respectivement, du syst\`eme d\'eflat\'e d'\'epaisseur $\ell$)
d\'efinie en~(\ref{eq_gammafx}),
nous montrons, voir le th\'eor\`eme~\ref{estim_gammal}, que l'in\'egalit\'e
$$\gamma_\ell\le \ell+\gamma_0$$
subsiste dans une boule centr\'ee en la racine dont le rayon
est proportionnel \`a l'inverse du carr\'e de l'\'epaisseur.

Cet article peut \^etre vu comme une g\'en\'eralisation de
G-Lecerf-Salvy-Y~\cite{GLSY07}. Sous les hypoth\`eses
suppl\'ementaires d'un syst\`eme carr\'e ($s=n$) et d'une racine
multiple de dimension de plongement un (c'est-\`a-dire le rang de la matrice
jacobienne chute num\'eriquement de un), nous traitions le cas des
grappes de racines en utilisant num\'eriquement le th\'eor\`eme des
fonctions implicites. Plus pr\'ecis\'ement il existe une fonction
analytique $\varphi(x_1,\ldots,x_{n-1})$ telle que
$\zeta_n=\varphi(\zeta_1,\ldots,\zeta_{n-1})$ et donc $\zeta_n$ est
une racine de la fonction \`a une variable
$h(x_n)=f_n(\varphi(x_1,\ldots,x_{n-1}),x_n)$. En
appliquant~\cite{GLSY05} \`a la fonction $h(x_n)$, nous pouvons en
d\'eduire \`a la fois la multiplicit\'e de $\zeta_n$ et un algorithme
approximant rapidement la racine $\zeta_n$. Remarquons que ce
r\'esultat inclut le cas des z\'eros ``simples doubles'' \'etudi\'e
pr\'ec\'edemment par Dedieu et Shub~\cite{DS00}.
\section{ $\,$Relation avec d'autres travaux}
 
Le cas d'une variable et une \'equation a \'et\'e \'etudi\'e
intensivement~; la g\'en\'eralisation de l'op\'erateur classique de
Newton est due \`a Schr\"oder (\cite{schroder1870}, page
324). L'$\alpha$-th\'eorie est faite dans~\cite{GLSY05}, avec des
citations s\'electionn\'ees.\\

Le cas g\'en\'eral a \'et\'e \'etudi\'e soit d'un point de vue
purement symbolique soit d'un point de vue num\'erique. Nous n'allons
pas traiter le cas uniquement symbolique, en nous r\'ef\'erant par
exemple au livre de Cox, Little, O'Shea~\cite{cox05} pour les
  notions fondamentales et \`a l'article de Lecerf~\cite{lecerf02} pour
  la d\'eflation (voir le paragraphe consacr\'e \`a Ojika plus bas).\\

Un des pionniers de l'approche num\'erique est Rall~\cite{rall66}. Il
traite le cas particulier o\`u la racine multiple $\zeta$ satisfait
l'hypoth\`ese suivante~: il existe un indice $m$, d\'efini comme la
multiplicit\'e de $\zeta$, tel que la suite d'espaces construits
it\'erativement \`a partir de $N_1=\Ker Df(\zeta) $ par
$$N_{k+1}=N_k\cap \Ker Df^{k+1}(\zeta),\quad k=1:m-1$$ 
aboutisse \`a $N_m=\{0\}$. Il est alors possible de construire
it\'erativement un op\'erateur qui retrouve la convergence quadratique
locale. L'id\'ee consiste \`a projeter it\'erativement l'erreur
$x_0-\zeta$ sur les noyaux $N_k$ et leur orthogonal $N_k^\perp$.\\

Au m\^eme moment, l'id\'ee d'utiliser une variante de la m\'ethode de
Gauss-Newton afin d'approximer une racine multiple a \'et\'e examin\'e
par
Shamanskii~\cite{shamanskii67}. Mais l'algorithme ne converge quadratiquement vers la racine singuli\`ere que sous des hypoth\`eses tr\`es particuli\`eres.\\

D'autres techniques, dites d'extension, ont \'et\'e \'etudi\'ees, o\`u quelques hypoth\`eses sont faites sur la racine singuli\`ere. Par exemple si l'op\'erateur induit par la projection de $\Ker Df(\zeta)$ dans $\Ker (Df(\zeta)^*)^\perp$:
$$\dis~\pi_{(\Ker Df(\zeta)^*)^\perp}D^2f(\zeta)(z,\pi_{\Ker Df(\zeta)})$$
est inversible, alors $(\zeta,0)$ devient une racine r\'eguli\`ere d'un nouveau syst\`eme, dit \'etendu, poss\'edant $2n-r$ variables. Le syst\`eme \'etendu est b\^ati \`a partir du syst\`eme initial et d'une d\'ecomposition en valeurs singuli\`eres de la matrice jacobienne. Cette voie est d\'evelopp\'ee par Shen et Ypma~\cite{ShenYpma2005} et g\'en\'eralise une technique d'extension utilis\'ee par   Griewank~\cite{G85} dans le cas o\`u la chute de rang n'est que de un. Au d\'ebut des ann\'ees $60$ une s\'erie de papiers traitent purement num\'eriquement de l'approximation des racines multiples \`a l'aide de techniques semblables, voir~\cite{reddien1978},~\cite{reddien1979},~\cite{DK180},~\cite{DK280},~\cite{GO81},~\cite{DKK83},~\cite{kelleysuresh1983},~\cite{yamamoto1983}. Mais ni la g\'eom\'etrie du probl\`eme ni la notion de multiplicit\'e n'y sont introduites.\\

Ojika dans~\cite{ojika87} propose une m\'ethode appel\'ee de
\emph{d\'eflation} pour d\'eriver un syst\`eme r\'egulier d'un
singulier, m\^elant calculs symboliques et num\'eriques. C'est une
g\'en\'eralisation d'un algorithme pr\'ec\'edemment d\'evelopp\'e
dans~\cite{OWM83}. Cette recherche d'un syst\`eme r\'egulier
\'equivalent fait intervenir une \'elimination de Gauss mais aucune analyse n'en est donn\'ee, en particulier il n'y a aucune d\'etermination du rang num\'erique ni de relation avec le concept de multiplicit\'e.\\

Dans le cas
particulier important de la localisation d'un syst\`eme polynomial et
dans un esprit purement symbolique, Lecerf dans~\cite{lecerf02}
reprend cet algorithme de d\'eflation qui rend un syst\`eme r\'egulier
\emph{triangulaire}, avec une complexit\'e arithm\'etique dans~:
$$\mathcal{O}\left
  (n^3(nL+n^\Omega)\mu(\zeta)^2\log(n\;\mu(\zeta))\right )$$ o\`u $n$
est le nombre de variables, $\mu(\zeta)$ la multiplicit\'e, $3\le
\Omega<4$ et $L$ est la longueur d'un calcul d'\'evaluation du syst\`eme.\\

Leykin, Verschelde et Zhao exhibent dans~\cite{lvz06} une m\'ethode m\^elant d\'eflation et extension, fond\'ee sur l'observation suivante~: si le rang num\'erique est $r$, il existe une solution isol\'ee $(\zeta,\delta)\in \C^n\times\C^{r+1}$ du syst\`eme
  \begin{equation}\label{dfl_versch}
  Df(x)B\delta=0,\quad \delta^*h-1=0,
\end{equation}  
o\`u $B\in \C^{n\times (r+1)}$ et $h\in\C^{r+1}$ sont choisis al\'eatoirement. La multiplicit\'e de la  racine $(\zeta,\delta )$ du syst\`eme d\'eflat\'e et \'etendu chute strictement. Un pas de la m\'ethode consiste alors \`a ajouter les \'equations 
~(\ref{dfl_versch}). Ceci implique \`a chaque pas dans le pire des cas
un doublement du nombre des variables et des  \'equations. De plus la
d\'etermination du rang num\'erique, reposant sur un travail de
Fierro-Hansen~\cite{fierro_hansen_05}, n'est pas libre de
$\varepsilon$. Leur th\'eor\`eme affirme alors qu'il suffit
d'ex\'ecuter $\mu(\zeta)-1$ pas pour arriver \`a un syst\`eme
r\'egulier.\\

Les papiers de Dayton-Zeng~\cite{DZ05}, Dayton-Li-Zeng~\cite{DLZ11} et
Nan Li-Lihong Zhi~\cite{li2014} rel\`event de la m\^eme veine et
traitent le cas polynomial  puis analytique.  Plus r\'ecemment des cas
particuliers ont \'et\'e \'etudi\'ees par Nan Li et Lihong Zhi dans
plusieurs travaux~\cite{li121},~\cite{li2012}. Mais toutes ces
contributions ne fournisent qu'une analyse num\'erique superficielle
de leur algorithme.\\

La dualit\'e et le rapport avec les matrices de Macaulay constituent le c{\oe}ur th\'eorique des travaux de Mourrain~\cite{Mou97}, Mantzaflaris et Mourrain~\cite{MM11}, ou plus r\'ecemment de Hauenstein, Mourrain, Szanto~\cite{hauenstein_mourrain_szanto_16}.
Dans ce dernier travail ils proposent quand la racine est connue
et dans un contexte purement symbolique, un nouvel algorithme pour
d\'eterminer un syst\`eme r\'egulier \`a partir du syst\`eme initial.
Le principe est de param\'etrer les matrices de multiplication :  le syst\`eme r\'egulier obtenu poss\`ede  $\dis N+\frac{n(+1)}{2}$ \'equations et  $\dis \frac{n \delta( \delta-1)}{2} $ variables, o\`u $N$ est le nombre d'\'equations du syst\`eme initial,
$n$ le nombre de variables et $\delta$ le cardinal d'une base de l'anneau local.
Par  ce biais  
ils \'etudient \'egalement une m\'ethode proche de la notre
toujours en supposant connue  la racine singuli\`ere : ils ajoutent  les relations entre les colonnes des matrices jacobiennes de rang d\'efectif. 

 \section{ $\,$D\'etermination du rang num\'erique d'une matrice}\label{rank_sec}

Soient $s \geq n$ deux entiers, et $M$ une $s\times n$-matrice \`a
coefficients complexes, $U\Sigma V^*$ une d\'ecomposition en valeurs
singuli\`eres $\sigma_1\ge \ldots \ge \sigma_n$ de $M$.\\

Nous consid\'erons les fonctions sym\'etriques \'el\'ementaires des
$\sigma_i$~:

$$s_{k}=\sum_{1\le i_1<\ldots <i_k\le n}
\sigma_{i_1}\ldots\sigma_{i_k},\quad k=1:n.$$
\index{$s_k$}
En d'autres termes, les valeurs singuli\`eres sont les racines du
polyn\^ome $s(\lambda)$ de degr\'e $n$~:

$$ s(\lambda) = \prod_{i=1}^n (\lambda - \sigma_i) = \lambda ^n + \sum_{i=n-1}^{0} (-1)^{(n-i)} s_{n-i} \lambda ^i .$$
\index{$s(\lambda)$}
Par convention $s_0=1$~; remarquons que cette convention est
naturelle, en ce qu'elle autorise le traitement du cas o\`u toutes les
valeurs singuli\`eres sont nulles, ce qui signifie que la matrice $M$
est nulle et donc que le rang l'est aussi.\\

Notre analyse du rang num\'erique est fond\'ee sur l'introduction des quantit\'es d\'efinies ci-dessous.
\begin{defin} \label{bkgkak}
 \quad\\
\begin{itemize}
\item[1--] $\dis b_{k}(M):=\sup_{0\le i\le k-1}
\left (\frac{s_{n-i}}{s_{n-k}}\right )^{\frac{1}{k-i}}$~
\index{$b_k(M)$},  $k=1:n$;
\vspace{0.2cm}	

\item[2--] $\dis g_{k}(M):=\sup_{k+1\le i\le n} 
\left (\frac{s_{n-i}}{s_{n-k}}\right )^{\frac{1}{i-k}}$~\index{$g_k(M)$},  $k=1:n-1$ et $g_n(M)=1$;
\vspace{0.3cm}

\item[3--] $\dis a_{k}(M):=b_{k}(M)\, g_{k}(M)$\index{$a_k(M)$},  $k=1:n$.
\end{itemize}
\end{defin}
\begin{req}
Les quantit\'es $b_k$, $g_k$ et $a_k$ sont un cas particulier
de celles introduites dans ~\cite{GLSY05} page 261.
Le choix de noter $b_k$  plut\^ot que 
$b_{n-k}$ est justifi\'e par l'identit\'e
$\dis\frac{s_n}{s_{n-1}}~=~\frac{s(0)}{s'(0)}$ obtenu pour $k=1$ car ce rapport est
 not\'e $\beta_1$ dans ~\cite{GLSY05}.
\end{req}
En fait la d\'etermination des $b_k(M)$ et $g_k(M)$ est
donn\'ee par le r\'esultat ci-dessous.
 \begin{pp}
 On a :
\begin{itemize}
\item[1--] $\dis b_{k}(M)=
\frac{s_{n-k+1}}{s_{n-k}}$, $k=1:n$;
\vspace{0.2cm}
\item[2--] $\dis g_{k}(M)=
\frac{s_{n-k-1}}{s_{n-k}}$, $k=1:n-1$.
\end{itemize}
\end{pp}
\begin{proof}
C'est une cons\'equence du th\'eor\`eme 5.2 de ~\cite{yak90} qui \'enonce que :
\\
\textit{Soient $r_0$ et $r_1$ tels que $nr_1-(n-1)r_0\ge 0$. Nous consid\'erons la suite $r_k=kr_1-(k-1)r_0$ pour $k\ge 2$. Tout polyn\^ome $
\dis f(x)=\sum_{k=0}^na_{n-k}x^{k}$ qui n'a que des racines r\'eelles
v\'erifie :
$$\frac{r_{n-k}}{r_{n-k-1}}\frac{k}{k+1}a_{n-k}^2-a_{n-k-1}a_{n-k+1}\ge 0,\quad k=1:n-1.$$
}
Avec $r_0=n$ et $r_1=n-1$ nous avons $r_i=n-i$. Les coefficients de $s(\lambda)$
v\'erifient donc :
$$\frac{i^2}{(i+1)^2}s_{n-i}^2-s_{n-i-1}s_{n-i+1}\ge 0,
\quad i=1:n-1.$$
Il s'ensuit que $s_{n-i}^2-s_{n-i-1}s_{n-i+1}\ge 0,\quad i=1:n-1.$ C'est \`a dire
$$\frac{s_{n-i-1}}{s_{n-i}}\le \frac{s_{n-i}}{s_{n-i+1}}
,\quad i=1:n-1.$$

Pour $k=1:n$ et $i=0:k-1$ il vient
$$\frac{s_{n-i}}{s_{n-k}}=\frac{s_{n-i}}{s_{n-i-1}}\frac{s_{n-i-1}}{s_{n-i-2}}\ldots
\frac{s_{n-k+1}}{s_{n-k}}\le \left(\frac{s_{n-k+1}}{s_{n-k}}\right)^{ k-i}.$$
Donc $\dis b_k(M)=\frac{s_{n-k+1}}{s_{n-k}}$. De la m\^eme fa\c con
nous obtenons la deuxi\`eme partie.
\end{proof} 
Nous allons pr\'eciser la notion de $\varepsilon$-rang que nous
utiliserons dans la suite.

\begin{defin}\label{defi_numerical_rank}
Soit $\varepsilon$ un nombre positif ou nul. Une matrice $M$ a un
$\varepsilon$-rang\index{$\varepsilon$-rang} \'egal \`a $r_{\varepsilon}$
si ses valeurs singuli\`eres v\'erifient~:
\begin{equation}\label{test_rank}
\sigma_1\ge \ldots \ge \sigma_{r_{\varepsilon}}>\varepsilon\ge \sigma_{r_{\varepsilon}+1}
\ge\ldots \ge \sigma_n.
\end{equation}
\end{defin}

Observons que le $\varepsilon$-rang est born\'e sup\'erieurement par le
rang $r$ lui-m\^eme.\\

Soit $\Sigma_\varepsilon$ la matrice obtenue \`a partir de $\Sigma$ en
mettant les $\sigma_ {r+1}, \ldots , \sigma_n$ \`a $0$. D\'efinissons la matrice
$M_\varepsilon $ comme $U\Sigma_\varepsilon V^*$.
\begin{req}
Si le rang de $M$ est au moins $r$, nous savons que $M_\varepsilon$
\index{$M_\varepsilon$} est la matrice de rang $r$ la plus proche de $M$.
\end{req}

\begin{req}
La d\'efinition~\ref{defi_numerical_rank} est justifi\'ee par le th\'eor\`eme de 
Eckart-Young-Mirsky~\cite{EY36}, \cite{mirsky60} qui poss\`ede une
longue histoire en th\'eorie de l'approximation de rang faible (voir
Markovsky~\cite{markovsky2011} pour des d\'eveloppements r\'ecents).
\end{req}
Par simplicit\'e nous noterons $a_k$, $b_k$, $g_k$ les valeurs
correspondantes $a_k(M)$, $b_k(M)$, $g_k(M)$.
\begin{tm}\label{mroots_tm}
Consid\'erons le polyn\^ome $s(\lambda)$ d\'efini pr\'ec\'edemment.
\begin{itemize}
 \item[1--]
S'il existe un entier $m$, compris entre $1$ et $n$, avec $a_{m}<1/9$,
alors le polyn\^ome $s(\lambda)$ poss\`ede $m$ racines dans la boule $B(0,
\varepsilon)$, o\`u~:
$$\varepsilon :=\frac{3a_{m}+1-\sqrt{(3a_{m}+1)^2-16a_{m}}}{4g_{m}}~;$$
 \item[2--] Si $a_1>1/9$ alors $\dis  \sigma_n>\frac{1}{10 g_m}$
o\`u $m$ est l'entier satisfaisant $s_{n-k}=0$ , $k=1:m-1$ et
$s_{n-m}\ne 0$.  
\end{itemize}
\end{tm}

\begin{proof} Prouvons la premi\`ere des assertions.
Comme $a_m<1/9$, la quantit\'e $s_{n-m}$ n'est pas nulle car elle
est strictement positive. Consid\'erons les polyn\^omes 
$$p(\lambda)=\frac{1}{s_{n-m}}s(\lambda)=\frac{1}{s_{n-m}}\prod_{i=1 }^n(\lambda-\sigma_i)=
\sum_{i=0}^{n}(-1)^{n-i}\frac{s_{n-i}}{s_{n-m} }\lambda^{i}$$
\index{$p(\lambda)$}
et
$$q(\lambda)= \sum_{i=m}^{n}(-1)^{n-i}\frac{s_{n-i}}{s_{n-m}}\lambda^{i}.$$
\index{$q(\lambda)$}
\begin{lm}\label{q}
Posons $t:= g_m |\lambda|$.	
\index{$t$} Alors pour tout $\lambda$ tel que $|\lambda| <
1/g_m$, donc pour tout $t<1$:
$$ |q(\lambda)| \ge |\lambda|^m \frac{1-2t}{1-t} $$
\end{lm}
\begin{proof}
\begin{align}
|q(\lambda)|&=\left |
\lambda^m+ \sum_{i=m+1}^{n}(-1)^{n-i}\frac{s_{n-i}}{s_{n-m}}\lambda^{i}
\right |\nonumber
 \\&\ge 
 |\lambda|^m- \sum_{i=m+1}^{n}\frac{s_{n-i}}{s_{n-m} } |\lambda|^{i} \nonumber
 \\
 &
\ge |\lambda|^m\left (1- \sum_{i=m+1}^{n}\frac{s_{n-i}}{s_{n-m} } |\lambda|^{i-m}\right )
 \nonumber
 \\
 &
\ge |\lambda|^m\left (1- \sum_{i\ge m+1 } (g_m |\lambda|)^{i-m} \right )
\nonumber\\
&
\ge |\lambda|^m\frac{1-2g_m|\lambda|}{1-g_m|\lambda|}. 
\end{align}
\end{proof}
Nous prouvons d'abord que $0$ est la seule racine de $q(\lambda)$
dans la boule ouverte $\dis B\left (0,\frac{1}{2g_m}\right )$.
Soit $\nu$ une racine non nulle de $q(\lambda)$. 
Alors nous avons par le lemme \ref{q} $$0=q(\nu)=|q(\nu)| \ge |\nu|^m \frac{1-2g_m|\nu|}{1-g_m|\nu|}.$$
Donc $\dis |\nu | \ge \frac{1}{2g_m}$.
\\
Consid\'erons le trin\^ome 
\begin{equation}\label{pol2}
2t^2-(3a_m+1)t+2a_m.
\end{equation}
Si $a_m <1/9$, alors ce trin\^ome a deux racines r\'eelles $t_1
<t_2$, car le discriminant 
$$\Delta = (3a_m+1)^2 - 16a_m = 9a_m{^2} - 10a_m +1 = (9a_m -
1)(a_m-1)$$ 
est strictement positif. Nous pouvons v\'erifier explicitement que
$t_1$ est strictement positif, puisque ceci se ram\`ene \`a $a_m$
strictement positif.

Nous prouvons que pour $|\lambda|$ satisfaisant
$\dis\frac{t_1}{g_m} \le |\lambda| < \frac{1}{2g_m}$, $p(\lambda)$
a $m$ racines, compt\'ees avec multiplicit\'e, dans la boule ouverte
$B(0,|\lambda|)$ (notons que la longueur de l'intervalle o\`u
$|\lambda|$ vit est strictement positif, puisque $t_1 < 1/2$).
Afin d'\'etablir ce fait, nous allons v\'erifier que l'in\'egalit\'e de Rouch\'e
\begin{equation}\label{eq_rouche_rank}
|p(\lambda)-q(\lambda)|<|q(\lambda)|
\end{equation}
est v\'erifi\'ee sur la sph\`ere de rayon $|\lambda|$.
Nous avons 
 \begin{align}\label{eq_rouche_rank_g}
 |p(\lambda)-q(\lambda)|&\le 
 \sum_{i=0 }^{m-1}\frac{s_{n-i}}{s_{n-m} }|\lambda|^{i}
 \nonumber
 \\
 &
\le \sum_{i=0 }^{m-1} b_m^{m-i}|\lambda|^{i}
\nonumber\\
 &
 \le |\lambda|^m\frac{b_m/|\lambda|}{1-b_m/|\lambda|}
\nonumber \\
 &
 \le \frac{a_m}{g_m|\lambda|-a_m}|\lambda|^m.
 \end{align}
Nous v\'erifions que $\dis t-a_m > t_1 - a_m = \frac{-a_m +1
   -\sqrt{\Delta}}{4}$ est strictement positif si $a_m<1/9$.\\

De~(\ref{eq_rouche_rank_g}) et du lemme~\ref{q}, nous voyons que
l'in\'egalit\'e de Rouch\'e est satisfaite si
 $$\frac{a_m}{t-a_m}|\lambda|^m
< \frac{1-2t}{1-t}|\lambda|^m.$$
Comme $|\lambda|$, $1-t$ et $t-a_m$ sont strictement positifs,
cette in\'egalit\'e est \'equivalente au trin\^ome (\ref{pol2}) n\'egatif, ce qui
est assur\'e par la condition $a_m<1/9$.\\
Donc sous la condition $a_m<1/9$ le polyn\^ome $p(\lambda)$ a exactement
$m$ racines compt\'ees avec multiplicit\'e dans la boule ouverte
$B(0,|\lambda|)$ o\`u $$ \varepsilon:=\frac{t_1}{g_m}\le |\lambda| <
\frac{1}{2g_m}.$$ 
Par cons\'equent nous avons 
$$\sigma_1\ge\ldots \ge \sigma_{n-m}>\varepsilon\ge \sigma_{n-m+1}\ge \ldots\ge \sigma_n.$$
Prouvons maintenant l'assertion $2$. De l'hypoth\`ese nous d\'eduisons
  que $s_n\ne 0$ puisque $a_1>1/9$. Le polyn\^ome $s(\lambda)$ s'\'ecrit 
$$s(\lambda)=s_n+(-1)^{n-m}s_{n-m}\lambda^m+\ldots -s_{n-1}\lambda^{n-1}+\lambda^n$$
avec $s_{n-m}\ne 0$. Nous avons~:
\begin{align*}
 \left |\frac{s(\lambda)}{s_{n-m}}\right |&\ge  \frac{s_n}{s_{n-m}}-\sum_{k=m}^n\frac{s_{n-k}}{s_{n-m}}|\lambda |^k
\\&
\ge b_m^m-|\lambda |^m\sum_{k=m}^n(g_m|\lambda |)^{k-m}
\\&
\ge b_m^m-\frac{|\lambda |^m}{1-g_m|\lambda |}
\\&
> \frac{1}{(9g_m)^m}-\frac{10}{(9(10g_m)^m}\qquad\textsf{puisque
$9b_mg_m\ge 1$ et pour $|\lambda |$ tel que $10|\lambda |g_m<1$}
\\
&>\frac{10^m-10\times 9^{m-1}}{9\times (90g_m)^m}
\\&
>0.
\end{align*}
Donc le polyn\^ome n'a pas de racine dans la boule $\dis B\left (
0,\frac{1}{10g_m}\right)$. Nous en concluons que $\sigma_n>\frac{1}{10g_m}.$	
\end{proof}
Une cons\'equence du th\'eor\`eme~\ref{mroots_tm} suit~:

\begin{tm}\label{rank_tm}
Soit $M$ une matrice.
\begin{itemize}
\item[1--] S'il existe un entier $m'$ ($1\le m' \le n$) avec
  $a_{m'}<1/9$, 
  soit $m$ le plus petit des entiers
    compris entre $1$ et $n$ avec $a_m<1/9$.
Alors la matrice $M$ a un $\varepsilon$-rang $n-m$ avec 
$$\varepsilon=\frac{3a_{m}+1-\sqrt{(3a_{m} +1)^2-16a_{m}}}{4g_{m}}.$$
 \item[2--] Si $a_1 >1/9$ alors le $\epsilon$-rang de la matrice $M$
   ($\epsilon=\frac{1}{10g_m}$) est $n$ o\`u $m$ est
   l'entier d\'efini dans l'assertion 2 du th\'eor\`eme~\ref{mroots_tm}.
\end{itemize}
\end{tm}
 \begin{tm}
L'algorithme de la table~\ref{table_rank} calcule le
$\varepsilon$-rang d'une matrice gr\^ace au th\'eor\`eme~\ref{rank_tm}. 
 \end{tm}
\begin{req}
En fait cet algorithme est libre de $\varepsilon$ et nous appellerons
le $\varepsilon$-rang ainsi calcul\'e le rang \emph{num\'erique} de la matrice.
\end{req}
  \begin{table}
 $$\fbox{
 \begin{minipage}{1\textwidth }
$$
\textbf{Rang num\'erique}$$
\begin{enumerate}[1-]
\item \quad Entr\'ee~: une matrice $M\in \C^{s\times n}$, $s\ge n$
\item \quad Calculer une approximation des valeurs singuli\`eres de $M$ : $\sigma_1\ge\ldots
  \ge\sigma_n$
\item  \quad De ces $\sigma_i$, calculer les quantit\'es $a_k$, $
  k=1:n$ et $g_k$ d\'efinis dans la section~\ref{rank_sec}
\item \quad S'il existe un $m' \ge 1$ tel que $a_{m'}<1/9$,
  soit $m$ le plus petit des entiers 
    compris entre $1$ et $n$ avec $a_m<1/9$. D\'efinissons
\\
\item \quad \quad\quad $\dis \varepsilon:=\frac{3a_m+1-\sqrt{(3a_m+1)^2-16a_m}}{4g_m}$
\item \quad \quad\quad Le $\varepsilon$-rang de la matrice $M$ est
  $n-m$, \quad\quad\textsf{de l'assertion 1 du th\'eor\`eme~\ref{rank_tm}}
\item \quad sinon
\item \quad \quad\quad $\varepsilon<\sigma_n$. Le $\varepsilon$-rang
  de la matrice $M$ est $n$,
 \\\qquad\qquad o\`u $\epsilon=\frac{1}{10g_m}$ comme dans
   l'assertion 2
   du th\'eor\`eme~\ref{rank_tm}
\item  \quad fin si
\item \quad Sortie~: le $\varepsilon$-rang de la matrice $M$
\end{enumerate}
   \end{minipage}
 }
 $$
 \caption{}\label{table_rank}
 \end{table} 
 
\section{ $\,$Le cadre fonctionnel}\label{functional}
Soient $n\ge 2$, $R_\omega\ge0$\index{$R_\omega$}
et $\omega\in \C^n$. \index{$\omega$} Nous consid\'erons l'ensemble
$\A2(\omega,R_{\omega})$
\index{$\A2(\omega,R_{\omega})$} des fonctions analytiques de carr\'e
int\'egrable dans la boule ouverte $B(\omega,R_{\omega})$. C'est un
espace de Hilbert \'equip\'e du produit int\'erieur 
$$<f,g> =\frac{c_n}{R_\omega^{2n}}
\int_{B(\omega,R_{\omega})}f(z)\overline{g(z)} dz,$$
o\`u $\dis c_n=\frac{n!}{\pi^n}$. Nous normalisons ce produit hermitien
en divisant l'int\'egrale par le volume de la boule $B(\omega,R_{\omega})\subset \C^n$. 
\\
Ensuite nous munissons $(\A2(\omega,R_{\omega}))^s$ d'une structure
hermitienne via le produit int\'erieur
$$<f,g>=\sum_{i=1}^s<f_i,g_i>.$$
Par simplicit\'e nous noterons $||f||$\index{$"|"|f"|"|$} la norme associ\'ee,
indiff\'eremment dans $\A2(\omega,R_{\omega})$ ou dans
$(\A2(\omega,R_{\omega}))^s$. De m\^eme, nous utilisons
la m\^eme notation $||.||$  pour la norme euclidienne de $\C$ ou $\C^s$.\\

Observons que ce cadre inclut le cas d'un syst\`eme analytique obtenu en
localisant un syst\`eme polynomial.
\subsection{Le noyau de Bergman}
Pour des r\'ef\'erences de base nous renvoyons \`a W.~Rudin~\cite{rudin08}
et S. G. Krantz~\cite{krantz13}.\\

Comme pour chaque $x\in B(\omega,R_{\omega})$ et chaque
$f\in\A2(\omega,R_{\omega})$ l'application \'evaluation $f\mapsto f(x)$
est une fonctionelle lin\'eaire continue $eval_x$ sur $\A2$, en
appliquant le th\'eor\`eme de repr\'esentation de Riesz, il existe un \'el\'ement $h_x \in
\A2$ tel que $$f(x)=eval_x(f)=<f,h_x>.$$\\
Posons $\dis \nu := x \mapsto \nu_x=\frac{\|x-\omega \|}{R_\omega}$.
\index{$\nu_x$}
\begin{defin}\label{bk}
La fonction $(z,x) \mapsto H(z,x):=\overline{h_x(z)}$\index{$H(z,x)$}
est appel\'ee le \emph{noyau de Bergman} et poss\`ede la propri\'et\'e
reproduisante~:
$$f(x)=\frac{c_n}{R_\omega^{2n}}\int_{B(\omega,R_{\omega})}f(z)\, H(z,x)\,dz,\quad \forall
f\in \A2(\omega,R_{\omega}).$$
Nous disons que le noyau de Bergman reproduit
$\A2(\omega,R_{\omega})$~; nous en \'enon\c cons quelques propri\'et\'es.
\end{defin}
\subsection{Propri\'et\'es}
\begin{pp}\label{berg_ker}
\quad
\begin{itemize}
\item[1--] $\dis H(z,x)=
\frac{1}{\dis (1-\frac{<z-\omega,x-\omega>}{R_\omega^2})^{n+1}}$;
\\
\item[2--]  $ \dis H(x,x)=\|H(\bullet,x)\|^2=
\frac{1}{(1-\nu_x^2)^{n+1}};$
\\
\item[3--] Pour tout $f\in \A2(\omega,R_{\omega})$ nous avons 
$$\dis |f(x)|=\frac{c_n}{R_\omega^{2n}}  \left |\int_{B(\omega,R_{\omega})}f(z)H(z,x)dz\right |
\le \frac{\|f\|}{(1-\nu_x^2)^\frac{n+1}{2}}.$$
\end{itemize}
\end{pp}
\begin{proof}
Voir Theorem 3.1.3. page 37 dans~\cite{rudin08}.
\end{proof}
La proposition ant\'ec\'edente se g\'en\'eralise aux d\'eriv\'ees d'ordre sup\'erieur.
  \begin{pp}\label{DkFH}
Soient $k\ge 0$, $\omega\in\C^n$, $x\in B(\omega,R_{\omega})$ et $u_i\in
\C^n$, $i=1:k$. Introduisons
$$H_k(z,x,u_1,\ldots, u_k)=
\frac{(n+1)\cdots (n+k)<z-\omega,u_1 >\cdots
<z-\omega, u_{k}>}{\dis R_\omega^{2k}\left (1-\frac{<z-\omega,x-\omega>}{R_\omega^2}
\right )^{k}}H(z,x).$$
Nous avons
\begin{itemize}
\item[1--] $\dis D^{k}f(x)(u_1,\cdots,u_k)= \frac{c_n}{R_\omega^{2n}}
\int_{B(\omega,R_{\omega})}f(z)\, H_k(z,x,u_1,\cdots, u_k)\, dz$ ; 
\\\\
\item[2--]  $\dis
\|D^kf(x)\|\le 
||f||\, 
\frac{(n+1)\cdots (n+k)}{ R_{\omega}^{k}\,(1-\nu_x^2)^{\frac{n+1}{2}+k}}$.
\end{itemize}
(\'evidemment si $k=0$ l'intervalle o\`u vit $i$ est vide, et les
produits\\
~$(n+1)\cdots (n+k)$ et $<z-\omega,u_1 >\cdots <z-\omega, u_{k}>$ sont
r\'eduits \`a $1$.)
\end{pp}
Pour prouver ceci nous avons besoin du lemme suivant~:
\begin{lm}\label{HK} 
$$\|H_k(\bullet,x,u_1,\ldots,u_n)\|
\le
 \frac{(n+1)\ldots (n+k)}{R_{\omega}^{k}(1-\nu_x^2)^{ \frac{n+1}{2}+k}}\|u_1\|\ldots \|u_k\|.$$
\end{lm}
\begin{proof}
Nous devons calculer l'int\'egrale de $H_k\bar H_k$ sur la boule
$B(\omega,R_{\omega})$. Ceci se r\'eduit \`a estimer 
$$I_k=\dis \frac{c_n}{R_\omega^{2n}}\int_{B(\omega,R_{\omega})}\frac{1}{\dis \left(1-
\frac{<z-\omega,x-\omega>}{R_\omega^{2}}\right )^{n+1+k}
\,\left(1-
\frac{\overline{<z-\omega,x-\omega>}}{R_\omega^{2}}\right )^{n+1+k}}dz$$
puisque
\begin{align*}
\|H_k(z,x,u_1,\ldots,u_n)\|&\le 
\frac{(n+1)\ldots(n+k)}{R_{\omega}^{k}}\|u_1\|\ldots \|u_k\| \, I_k^{1/2}.
\end{align*}
Nous avons
\begin{align*}
\dis I_k&=\dis \frac{c_n}{R_\omega^{2n}}\int_{B(\omega,R_{\omega})}H(z,x)\frac{1}{\dis \left(1-
\frac{<z-\omega,x-\omega>}{R_\omega^{2}}\right )^{k}
\,\left(1-
\frac{\overline{<z-\omega,x-\omega>}}{R_\omega^{2}}\right )^{n+1+k}}dz\\
&= 
\frac{1}{(1-\nu_x^2)^{n+1+2k}}
\end{align*} 
en utilisant la formule du noyau de Bergman (Proposition
\ref{berg_ker}) et sa propri\'et\'e de reproduction appliqu\'ee \`a la fonction
$\dis
z\mapsto
\frac{1}{\dis \left(1-
\frac{<z-\omega,x-\omega>}{R_\omega^{2}}\right )^{k}
\,\left(1-
\frac{\overline{<z-\omega,x-\omega>}}{R_\omega^{2}}\right )^{n+1+k}}.$

Il s'ensuit la preuve du lemme.
\end{proof}
Nous d\'emontrons maintenant la proposition~\ref{DkFH}.
 \begin{proof}
Nous proc\'edons par r\'ecurrence. La proposition~\ref{berg_ker} r\`egle le
cas $k=0$. Ensuite nous avons~:
\begin{align*}
D^{k+1}f(x)&(u_1,\ldots,u_k,u_{k+1})=\left .
\frac{d}{dt}D^kf(x+tu_{k+1})(u_1,\ldots,u_k)\right |_{t=0}
\\
&=\left .\frac{d}{dt}\frac{c_n}{R_\omega^{2n}}\int_{B(\omega,R_{\omega})}f(z) H_k(z,x+tu_{k+1},u_1,\ldots,u_k)dz\, \right |_{t=0} 
\\
&=\frac{c_n}{R_\omega^{2n}}\int_{B(\omega,R_{\omega})}f(z)\frac{ H_k(z,x,u_1,\ldots ,u_k)(n+1+k) 
<z-\omega, u_{k+1}>}{\dis R_{\omega}^2\left (1-\frac{<z-\omega,x-\omega>)}{R_{\omega}^2}\right )}dz
\\
&=\frac{c_n}{R_\omega^{2n}}\int_{B(\omega,R_{\omega})}f(z) H_{k+1}(z,x,u_1,\ldots,  u_{k+1}) dz.
\end{align*}
D'o\`u la preuve de la premi\`ere assertion. Pour la seconde nous \'ecrivons
$$\|D^kf(x)(u_1,\ldots,u_k)\|\le \| f\|\,\|H_k(\bullet,x,u_1,\ldots ,u_k)\|.$$
Nous concluons en utilisant le lemme~\ref{HK}.
\end{proof}
Des propositions~\ref{berg_ker} et~\ref{DkFH} nous d\'eduisons ais\'ement que
\begin{pp}\label{Berg_F_DFK}
Pour tout $k\ge 0$, $x\in\C^n$ et $f\in (\A2(\omega,R_{\omega}))^s$
nous avons
\begin{equation*}
\dis \|D^kf(x)\|\le ||f||\, 
\frac{(n+1)\ldots (n+k)}{R_{\omega}^k(1-\nu_x^2)^{\frac{n+1}{2}+k}}.
\end{equation*}
\end{pp}
\section{ $\,$Analyse de l'application \'evaluation}\label{evaluation_sec}
L'application \'evaluation est d\'efinie par
$$\eval\,:\, (f,x)\mapsto eval_x(f) = f(x)$$ 
de $(\A2(\omega,R_{\omega}))^s\times
B(\omega,R_{\omega})$ dans $\C^s$.\\
\index{$eval_x$}
Posons $\dis c_0:=\sum_{k \ge 0}(1/2)^{2^k-1}$ ($\sim 1.63...$), et
$\alpha_0$ ($\sim 0.13...$) la premi\`ere des racines positives du
trin\^ome $(1-4u+2u^2)^2-2u$.
\index{$c_0$}
\index{$\alpha_0$}
\\
Quand la valeur $f(x)$ peut-elle \^etre consid\'er\'ee comme petite~? Nous
allons en donner un sens pr\'ecis en calculant une valeur seuil.
\begin{tm}\label{evaluation_map_tm}
Soient $f=(f_1, \ldots, f_s)\in \A2(\omega,R_{\omega})^s$ et $x\in B(\omega,R_{\omega})$.
Si
$$ c_0 \, (1-\nu_x^2)^\frac{n+1}{2}\frac{\|f(x)\|}{R_{\omega}}+
\nu_x < 1$$
 et
$$
 \frac{(n+1)(n+2)}{2}(1-\nu_x^2)^{(n-1)/2} \left (
\frac{\|f\|}{\,R_{\omega}}\frac{1}{1-\nu_x^2}+1
 \right )\, \frac{\|f(x)\|}{R_{\omega}}\le \alpha_0
 $$
   alors $f(x)$ est petit dans le sens suivant~: la suite de Newton
   d\'efinie par
     $$(f^0,x_0)=(f,x), \quad (f^{k+1  },x_{k+1})=((f^{k},x_{k})-
 D\eval(f^{k},x_{k})^\dagger  \eval(f^{k},x_{k})),           \quad k\ge 0,$$
 en notant $D\eval(f^{k},x_{k})^\dagger$ l'inverse
   g\'en\'eralis\'e de Moore-Penrose de $D\eval(f^{k},x_{k})$,
 converge quadratiquement vers un certain 
  $\dis (g,y)\in   
  (\A2(\omega,R_{\omega}))^s\times B(\omega,R_{\omega})$ satisfaisant $g(y)~=~0$.
  Plus pr\'ecis\'ement nous avons
 $$( \|f-g\|+\|x-y\|^2)^{1/2}\le c_0 \, (1-\nu_x^2)^\frac{n+1}{2}\|f(x)\|.$$
\end{tm}
Il s'ensuit imm\'ediatement le corollaire~:
\begin{cl}\label{cl_evaluation}
Consid\'erons le cas particulier $x=\omega$ dans le th\'eor\`eme~\ref{evaluation_map_tm}.
Si
$$
c_0 \frac{\|f(x)\|}{R_\omega} < 1
$$
et
$$
\frac{(n+1)(n+2)}{2}  
\left ( \frac{\|f\|}{R_{\omega}} + 1 \right ) \, \frac{\|f(x)\|}{R_\omega}
\le \alpha_0 
$$
  alors $f(x)$ est petit. 
Plus pr\'ecis\'ement il existe $(g,y)\in (\A2(x,R_{\omega}))^s\times
B(x,R_{\omega})$ tel que $g(y)=0$ et $$( \|f-g\|+\|x-y\|^2)^{1/2}\le c_0\|f(x)\|.$$
\end{cl}
La suite de cette section consiste \`a \'etablir le th\'eor\`eme~\ref{evaluation_map_tm}.
\subsection{Estimation des d\'eriv\'ees de l'application \'evaluation}
\begin{pp}\label{Deval_pseudo}
$$\| D\eval (f,x)^\dagger\|\le 
(1-\nu_x^2)^{\frac{n+1}{2}}.$$
\end{pp}
\begin{proof}
La d\'eriv\'ee de l'application \'evaluation est
$$D\eval(f,x)(g,y)=g(x)+Df(x)y.$$
Donc $(g,y)\in \ker D\eval(f,x)$ si et seulement si $g(x)+Df(x)y=0$,
c'est-\`a-dire 
$$ <g_i,H(\bullet,x)>+<y,Df_i(x)^*>=0,\quad i=1:s.$$
Cette condition peut \^etre exprim\'ee \`a l'aide du produit int\'erieur de
$(\A2)^s\times \C^n$~:
$$<g,(0,\ldots,0,H(\bullet,x),0,\ldots,0)>+<y,Df_i(x)^*>=0,
\quad i=1:s.$$
Donc l'espace vectoriel $(\ker D\eval (f,x))^\perp$ est engendr\'e par
l'ensemble 
$$(H(\bullet ,x)v, Df(x)^*v)$$
o\`u $v\in\C^n$. La condition
$$D\eval (f,x)(H(\bullet,x),\,Df(x)^*v)=u$$
devient
$$\left(H(x,x)I_{s}+Df(x)Df(x)^*\right )v=u.$$
La matrice $\mE=H(x,x)I_{s}+Df(x)Df(x)^*$ est la somme d'une matrice
diagonale positive  et d'une matrice hermitienne. Appliquant le
th\'eor\`eme de Weyl page 203 dans \cite{stewart_sun90}, les valeurs propres de la
matrice $\mE$ sont sup\'erieures \`a celles de  $H(x,x)I_{s}>0$.
Donc la norme de la matrice inverse $\mE^{-1}$ v\'erifie
$$\|\mE^{-1}\|\le \frac{1}{H(x,x)}.$$
Ceci permet de calculer $\|D\eval (f,x)^\dagger\|$. En fait, soit
$u,v\in \C^n$ tel que $\mE v=u$. Nous avons
\begin{align*}
\|D\eval (f,x)^\dagger u\|^2&=\|H(\bullet,x)\|^2\,\|v\|^2+\|Df(x)v\|^2
\\
&=H(x,x)\, \|v\|^2+\|Df(x)^*v\|^2.
\end{align*}
Comme la matrice $\mE^{-1}$ est hermitienne, nous pouvons \'ecrire
\begin{align*}
\|D\eval (f,x)^\dagger u\|^2&=
v^*\mE v
\\
&=u^*\mE^{-1} u
\\
&\le \|\mE^{-1}\|\, \|u\|^2.
\end{align*}
Finalement
\begin{align*}
\|D\eval (f,x)^\dagger \|^2 &\le \|\mE^{-1}\|
\\
&
\le \frac{1}{H(x,x)}
\\
&\le
(1-\nu_x^2)^{n+1},\qquad\textrm{par la proposition~\ref{berg_ker} }.
\end{align*}
Ceci ach\`eve la preuve de la proposition.
\end{proof}	
\begin{pp}\label{Dkeval}
$$\|D^k\eval (f,x)\|\le 
\frac{(n+1)\ldots (n+k)\,\|f\|}{R_{\omega}^{k}(1-\nu_x^2)^{\frac{n+1}{2}+k}}
+\frac{k(n+1)\ldots (n+k-1)\,}{R_{\omega}^{k-1}(1-\nu_x^2)^{\frac{n+1}{2}+k-1}}.
$$
\end{pp}
\begin{proof}
Nous avons
\begin{align*}
D^k\eval (f,x)&(g^{(1)},y^{(1)},\ldots ,g^{(k)},y^{(k)})
\\ &= D^kf(x)(y^{(1)},\ldots,y^{(k)})+\sum_{j=1}^k D^{k-1}g^{(j)}(x)(y^{(1)},\ldots,\widehat{y^{(j)}},\ldots, y^{(k)}),
\end{align*}
o\`u $\widehat{y^{(j)}}$ signifie que ce terme n'appara\^\i t pas. Alors en
utilisant la proposition~\ref{DkFH} nous trouvons que
\begin{align*}
\|D^k&\eval (f,x)(g^{(1)},y^{(1)},\ldots ,g^{(k)},y^{(k)})\|
\\ &\le \| D^kf(x)(y^{(1)},\ldots,y^{(k)})\|+\sum_{j=1}^k\| D^{k-1}g^{(j)}(x)(y^{(1)},\ldots,\widehat{y^{(j)}},\ldots, y^{(k)})\|
\\
&
\le 
\frac{(n+1)\ldots (n+k)\,\|f\|}{R_{\omega}^{k}(1-\nu_x^2)^{\frac{n+1}{2}+k}}\|y^{(1)}\|\ldots \|y^{(k)}\|
\\
&\quad\quad\quad 
 +\sum_{j=1}^k\frac{(n+1)\ldots (n+k-1)\,\|g^{(j)}\|\,}{R_{\omega}^{k-1}(1-\nu_x^2)^{\frac{n+1}{2}+k-1}}\|y^{(1)}\|\ldots \widehat{\|y^{(j)}\|}\ldots\|y^{(k)}\|
.
\end{align*}
Nous bornons $\|y^{(j)}\|$ et $\|g^{(j)}\|$ par
$\|(g^{(j)},y^{(j)})\|$. Nous obtenons
\begin{align*}
\|D^k&\eval (f,x)(g^{(1)},y^{(1)},\ldots ,g^{(k)},y^{(k)})\|
\\ & \le 
\left (\frac{(n+1)\ldots (n+k)\,\|f\|}{R_{\omega}^{k}(1-\nu_x^2)^{\frac{n+1}{2}+k}}
+\frac{k(n+1)\ldots (n+k-1)\,}{R_{\omega}^{k-1}(1-\nu_x^2)^{\frac{n+1}{2}+k-1}}
\right )\\
& 
\hspace{8cm}||(g^{(1)},y^{(1)})\|\ldots\|(g^{(k)},y^{(k)})\|
.
\end{align*}
Finalement 
$$\|D^k\eval (f,x)\|\le \frac{(n+1)\ldots (n+k)\,\|f\|}{R_{\omega}^{k}(1-\nu_x^2)^{\frac{n+1}{2}+k}}
+\frac{k(n+1)\ldots (n+k-1)\,}{R_{\omega}^{k-1}(1-\nu_x^2)^{\frac{n+1}{2}+k-1}}
.$$
\end{proof}
\subsection{D\'emonstration du th\'eor\`eme~\ref{evaluation_map_tm}}
La d\'emonstration utilise le th\'eor\`eme 128 page 121 de J.-P. Dedieu,
Points fixes, z\'eros et la m\'ethode de Newton, Springer, 2006.
\begin{tm}\label{dedieu_128}
Donnons-nous $f$ une application analytique de $\E$ dans $\F$, deux
espaces de Hilbert. Soit $x\in \C^n$. Nous supposons que la d\'eriv\'ee
$Df(x)$ est surjective. En notant par $Df(x)^\dagger$ l'inverse g\'en\'eralis\'e de Moore-Penrose
de $Df(x)$, nous introduisons les quantit\'es
\begin{itemize}
\item[1--] $\dis \beta(f,x)=\|Df(x)^\dagger f(x)\|$~;
\vspace{.2cm}

\item[2--] $\dis \gamma(f,x)=\sup_{k \ge 2}\|\frac{1}{k!}
Df(x)^\dagger D^kf(x)\|^{\frac{1}{k-1}}$~;
\vspace{.2cm}

\item[3--] $\dis \alpha(f,x)= \beta(f,x)\gamma(f,x)$.
\end{itemize}
Rappelons que $\alpha_0$ et $c_0$ sont les constantes introduites dans
cette section.
\\
Si $\alpha(f,x)\le \alpha_0 $ alors il existe un z\'ero $\zeta$ de $f$
dans la boule $B(x_0,c_0\beta(f,x_0))$ et la suite de Newton
$$x_0=x,\quad x_{k+1}=x_k-Df(x_k)^\dagger f(x_k),\quad k\ge 0,$$
converge quadratiquement vers $\zeta$.
\end{tm}
Nous sommes d\'esormais pr\^ets pour prouver le th\'eor\`eme~\ref{evaluation_map_tm}.
\begin{proof}
Elle consiste \`a v\'erifier la condition $\alpha(\eval,(f,x))\le\alpha_0$.
Utilisant les propositions~\ref{Deval_pseudo} et~\ref{Dkeval}, nous
pouvons borner la quantit\'e $\gamma(\eval,(f,x))$. Nous obtenons
\begin{align*}
\gamma(\eval,(f,x))&\le \sup_{k\ge 2} \left (\frac{1}{k!}
\, \|D\eval (f,x)^\dagger\|\, \|D^k\eval (f,x)\|\right )^{\frac{1}{k-1}}
\\
&\le
\sup_{k\ge 2} \left (\binomial{n+k}{k}\frac{\|f\|\, }{R_{\omega}^{k}(1-\nu_x^2)^{k}}
+
\binomial{n+k-1}{k-1}\frac{1}{R_{\omega}^{k-1}(1- \nu_x^2)^{k-1}} 
\right )^{\frac{1}{k-1}}.
 \end{align*}
Nous savons que $\dis \binomial{n+k}{k}=
 \frac{n+k}{k}\binomial{n+k-1}{k-1}$. De plus la fonction $\dis
 k\mapsto \binomial{n+k}{k}^{\frac{1}{k-1}}$ est d\'ecroissante. Donc
 $\dis \binomial{n+k}{k}^{\frac{1}{k-1}}\le
 \frac{(n+1)(n+2)}{2}$. Ainsi nous obtenons l'estimation au point
 \begin{equation}\label{gamma_eval_eq}
 \gamma(\eval,(f,x))\le
 \frac{(n+1)(n+2)}{2R_{\omega}(1-\nu_x^2)}
 \left (
 \frac{\|f\|\,}{R_{\omega}(1-\nu_x^2)
} +  1
\right )
 .
 \end{equation} 
De ma m\^eme fa\c con la quantit\'e $\alpha(\eval, (f,x))$ peut \^etre born\'ee par
  \begin{align*}
 \alpha(\eval,(f,x))&\le \gamma(\eval, (f,x))\, \beta(\eval,(f,x))
 \\
 &\le \gamma(\eval, (f,x))\, \|D\eval (f,x)^\dagger\|\,\|f(x)\|.
 \end{align*}
 En utilisant les in\'egalit\'es de la proposition~\ref{Deval_pseudo} et~(\ref{gamma_eval_eq})
 il vient
\vspace{0.2cm}

 \begin{align}\label{alpha_eval_eq}
\alpha&(\eval , (f,x)) \le \frac{(n+1)(n+2)}{2R_\omega}\,(1-\nu_x^2)^{(n-1)/2} \left (
\frac{\|f\|\,}{R_{\omega}(1-\nu_x^2)
} +  1
 \right )\, \|f(x)\|.
\end{align}
La condition $$\frac{(n+1)(n+2)}{2R_\omega}\,(1-\nu_x^2)^{(n-1)/2} \left (
\frac{\|f\|\,}{R_{\omega}(1-\nu_x^2)
} +  1
 \right )\, \|f(x)\|\le \alpha_0$$
 implique \'evidemment $\alpha(\eval (f,x))\le \alpha_0$.
 \\
Donc le th\'eor\`eme~\ref{dedieu_128} s'applique. La suite de Newton
 $$(f^0,x_0)=(f,x), \quad (f^{k+1  },x_{k+1})=((f^{k},x_{k})-
 D\eval(f^{k},x_{k})^\dagger  \eval(f^{k},x_{k}),           \quad k\ge 0,$$
converge vers un certain
  $\dis (g,y)\in  B((f,x),c_0\beta(\eval,(f,x))\subset
  (\A2(\omega,R_{\omega})^s\times \C^n$.
En d'autres termes
\begin{align*}
(\|f-g\|^2+\|x-y\|^2)^\frac{1}{2} &\le c_0\beta(\eval,(f,x))
\\
&
\le 
c_0 \|D\eval (f,x)^\dagger\|\,   \|f(x)\|
\\
&\le    c_0 \, (1-\nu_x^2)^\frac{n+1}{2}\|f(x)\|.
\end{align*}
Ceci implique que $y\in B(\omega,R_{\omega})$ parce que
 \begin{align*}
 \|y-\omega\|&\le \| y-x\|+\rho_x
 \\
 &\le  c_0 \, (1-\nu_x^2)^\frac{n+1}{2}\|f(x)\|+\rho_x
 \\
 &<R_{\omega}. \quad\quad \textsf{par hypoth\`ese}.
\end{align*} 
Ceci ach\`eve la preuve.
 \end{proof}
\section{ $\,$D\'eflation et op\'erateur de Newton singulier}\label{Kern-Sing-Newton}
Nous d\'efinissons dans cette section une suite de d\'eflation en point $x_0$ proche d'une racine $\zeta$. Comme nous l'avons \'evoqu\'e
en section~\ref{sec2} celle-ci est la combinaison d'une op\'eration de s\'election et d'une op\'eration de d\'enoyautage.  Si $x_0=\zeta	$ rappelons  que  nous commencons par remplacer les
\'equations par les d\'eriv\'ees d'ordre la valuation moins un.
Nous obtenons ainsi un syst\`eme de rang plus grand que un.
Puis, si le rang est plus petit que $n$,
nous
pr\'eparons ce syst\`eme en divisant les \'equations en deux
familles. L'invariant qui pr\'eside \`a cette partition est le rang $r$ de
la matrice jacobienne $Df(\zeta)$.
 Sans perte de g\'en\'eralit\'e nous pouvons supposer
que les $r$ premiers g\'en\'erateurs poss\`edent des parties affines
lin\'eairement ind\'ependantes. L'op\'eration de d\'enoyautage consiste \`a ajouter au $r$ premi\`eres  fonctions celles qui constituent le compl\'ement de Schur de $Df(x)$ associ\'e
\`a $D_{1:r}f_{1:r}(x)$.
En section~\ref{Mult-drops-kerneling} nous  montrons que la multiplicit\'e de la racine du syst\`eme obtenu apr\`es un cran de d\'eflation a chut\'e strictement. Comme nous l'avons soulign\'e en section~\ref{sec2}, nous pouvons alors r\'eit\'erer
ce proc\'ed\'e.

D\'etaillons le proc\'ed\'e de d\'eflation d\'ecrit ci-dessus quand $x_0$ est proche de $\zeta$. Celui-ci repose sur les deux propri\'et\'es suivantes. Premi\`erement l'\'evaluation en $x_0$ d'une fonction qui s'annule en $\zeta$ est petite au sens de l'analyse effectu\'ee \`a la section~\ref{evaluation_sec}. Deuxi\`emement le rang num\'erique de $Df(x_0)$
est celui de $DF(\zeta)$ pour un $\epsilon$ d\'etermin\'e par l'analyse   
effectu\'ee en section~\ref{rank_sec}. Pr\'ecisons ces id\'ees en commen\c cant par introduire
 la notion de valuation \`a $\varepsilon$ pr\`es.
\begin{defin}
Soient $\varepsilon\ge 0$, $x_0\in\C^n$ et $f\in\C\{x-x_0\}$.  Nous disons que
$f$ a  une $\varepsilon$-valuation $p$\index{$\varepsilon$-valuation} en $x_0$ si 
\begin{itemize}
\item[1--] $\dis \forall k<p,\forall \alpha\in \N^n$
tel que  $|\alpha |=k\quad \textrm{et} \quad \left |
\frac{\partial^kf(x_0)}{\partial x^\alpha}
\right |\le \varepsilon;$
\item[2--] $\dis \exists\, \alpha\in \N^n$ tel que $|\alpha |=p\quad \textrm{et} \quad \left |
\frac{\partial^pf(x_0)}{\partial x^\alpha}
\right |	 >\varepsilon.$
\end{itemize}
Si $\varepsilon=0$ la valuation est dite exacte.
\end{defin}
\begin{avt}
Les d\'efinitions qui vont suivre n\'ecessitent de consid\'erer les syst\`emes tant\^ot comme des listes, tant\^ot comme des vecteurs, tant\^ot comme des ensembles. Par exemple,
nous notons $\vect (\bullet)$ l'op\'erateur qui concat\`ene les \'el\'ements non nuls d'un ensemble ou d'une
matrice en un vecteur ligne.
\end{avt}
Nous d\'efinissons ci-dessous un op\'erateur de s\'election.
\begin{defin}
Soient $\varepsilon\ge 0$, $x_0\in\C^n$ et
$f=(f_1,\ldots,f_s)\in\C\{x-x_0\}^s$. Nous notons $p_k$ 
l'$\varepsilon$-valuation en $x_0$ de $f_k$, $k=1:s$. 
Soit $\dis \Delta_k=\left \{\frac{\partial^{p_k-1}f_k(x)}{\partial x^\alpha}\,:\,|\alpha|=p_k-1\right\}$.
Nous d\'efinissons l'op\'erateur de s\'election $S$ par
$$
\dis S\,:\, f\rightarrow\vect\left (\overset{s}{\underset{k=1}{\bigcup}}\Delta_k
\right).\index{$S(f)$}	
$$
Si $\varepsilon=0$ l'op\'eration de s\'election est dite exacte.
\end{defin}
\begin{req}
Le calcul de $S(f)$ s'effectue par l'algorithme r\'ecursif dit de s\'election en faisant  S\'election$(x_0,f,\emptyset,\emptyset)$ . Il est facile de voir que cet algorithme est libre de $\epsilon$ en supposant que le calcul de la norme dans $\A2(x_0,R_{x_0})$ le soit. De plus le nombre d'\'etapes de cet algorithme est fini et $S(f):=S_f$.
\end{req}
\begin{table}
$$\fbox{ \begin{minipage}{1\textwidth }
\textbf{Algorithme de s\'election} : S\'election$(x_0,f,S_f,S_1)$ 
\\avec~: $x_0\in \C^n$, $f\in\A2(x_0,R_{x_0})^s$, $S_f\in\A2(x_0,R_{x_0})^s$,
$S_1\in\A2(x_0,R_{x_0})^s$
\begin{enumerate}[1-]
\item \quad Pour $k=1$ \`a $\sharp f$
\item \quad\quad $\dis \eta:=\frac{2\alpha_0}{(n+1)(n+2)(R_{x_0}+\|f_k\|)\,R_{x_0}^{n-2}}$
\item \quad\quad Si $\|f_k(x_0\|\le \eta$ alors \quad\quad\quad
  \textsf{test justifi\'e par le corollaire~\ref{cl_evaluation}}
 \item \quad\quad\quad  $S_1:=\{f_k\}$
  \item \quad\quad \quad S\'election$(x_0,\{\nabla f_k\}\minus \{0\},S_f,S_1)$
\item \quad \quad Sinon
\item \quad  \quad \quad $S_f:= S_f\cup S_1$
\item \quad \quad fin Si
\item \quad fin Pour
\end{enumerate}
    \end{minipage}
 }
 $$
   \caption{}\label{selection_table}
 \end{table}

Comme nous utiliserons souvent la notion de \emph{compl\'ement de Schur}
dans la suite, nous rappelons sa d\'efinition.
\begin{defin}
Le compl\'ement de Schur d'une matrice $\dis M=\left (\begin{array}{cc}
A&B \\C&D \end{array}\right )$ de rang $r>0$ associ\'e \`a une
sous-matrice inversible $A$ de rang $r$ est par d\'efinition 
  $\schur(M):=D-CA^{-1}B$. 
  \\
  Si $r =0$ nous d\'efinissons $\schur(M):=M$.\index{$\schur(M)$}
  \end{defin}
\begin{defin}\label{defi_Kf}
Soient $\varepsilon\ge 0$, $0\le r<n$ et $f=(f_1,\ldots,
f_s)\in\C\{x-x_0\}^s$. Supposons que $D_{1:r}f_{1:r}(x_0)$ a un
$\varepsilon$-rang \'egal \`a $r$.
Nous d\'efinissons l'op\'erateur de \emph{d\'enoyautage}
$$ K\, :\,f\mapsto \left (f_1,\ldots,f_r, vec(\schur (Df(x)) )\right
)\in \C\{x-x_0\}^{r+(n-r)\, (s-r)}.
$$
Nous disons que $K(f)$ est un $\varepsilon$-d\'enoyautage de $f$ si nous avons
\begin{equation}\label{test_Kf}
\|K(f)(x_0)\|\le \varepsilon.
\end{equation}
Le d\'enoyautage est \emph{exact} quand $\varepsilon=0$.\index{$K(f)$}
\end{defin}  
\begin{defin}\label{dfl_seq} \textbf{$(${\bf Suite de d\'eflation}$)$.}
Soient $\varepsilon\ge 0$, $x_0\in \C^n$ et $f=(f_1,\ldots, f_s)\in
\C\{x-x_0\}^s.$ La suite
\begin{align*}
  F_0&=S(f)\\
  F_{k+1}&=S( \,K ( F_k)\, ), \quad k\ge 0,
 \end{align*}
est appel\'ee \emph{suite de d\'eflation}.
 \\ 
L'\emph{\'epaisseur} est par d\'efinition l'indice $\dis \ell$
\index{$\ell$} o\`u le $\varepsilon$-rang de $DF_\ell(x_0)$ est \'egal \`a
$n$, et pas avant. Nous verrons dans la section
\ref{Mult-drops-kerneling} que l'\'epaisseur est finie. Enfin nous
appelons \emph{syst\`eme d\'eflat\'e} $\dfl(f)$ de $f$\index{$\dfl(f)$} un
syst\`eme carr\'e de rang $n$ extrait de $F_\ell$.
 \end{defin}
 \begin{req}\label{req_Sk}
 Par souci de simplification nous avons not\'e les op\'erateurs de s\'election et de d\'enoyautage $S$ et $K$ respectivement plut\^ot que $S_{x_0,\epsilon}$ et
 $K_{x_0,\epsilon}$.
 \end{req}
 \begin{req}\label{req_rang_dfl}
 Par construction, le rang de chaque syst\`eme d'une suite de d\'eflation est non nul.
 \end{req}
 \begin{req}\label{req_rang_dfl2}
 Quand le rang num\'erique de la matrice jacobienne $DF(x_0)$ est nul, on peut remarquer qu'une \'etape de l'op\'eration de s\'election correspond \`a une op\'eration de d\'enoyautage. On verra qu'en section~\ref{sec-gammal}, l'analyse num\'erique de l'agorithme de calcul d'un syst\`eme d\'eflat\'e
 est simplifi\'ee si le rang num\'erique de $DF(x_0)$
 est strictement positif, ce qui est le cas apr\`es une op\'eration de s\'election.
 \end{req}
\textbf{Note historique}. Nous adoptons la terminologie
\emph{\'epaisseur} introduite par Emsalem dans~\cite{emsalem78} plut\^ot
que le terme \emph{profondeur ``depth''} utilis\'e plus r\'ecemment par Mourrain,
Matzaflaris dans~\cite{MM11} ou Dayton, Li, Zeng~\cite{DZ05},
~\cite{DLZ11}.~\,~$\circ$
\begin{tm}\label{dfl_tm}
Soient $x_0\in \C^n$ et $f\in\A2(x_0,R_{\omega})^s$. Alors l'algorithme
d\'ecrit dans la table~\ref{dfl_table} prouve l'existence d'une suite de
d\'eflation o\`u les tests de v\'erification des
in\'egalit\'es~(\ref{defi_numerical_rank}) and~(\ref{test_Kf}) sont ex\'ecut\'es
respectivement gr\^ace au th\'eor\`eme~\ref{rank_tm} et au corollaire~\ref{cl_evaluation}.
\end{tm}
\begin{table}
 $$\fbox{
 \begin{minipage}{1\textwidth }
\textbf{Suite de d\'eflation et syst\`eme d\'eflat\'e}
\begin{enumerate}[1-]
\item \quad Entr\'ees~: $x_0\in \C^n$, $f\in\A2(x_0,R_{x_0})^s$
\item \quad $F:=S(f)$.
\item \quad $\dis \eta:=\frac{2\alpha_0}{(n+1)(n+2)(R_{x_0}+\|F\|)R_{x_0}^{n-2}}$
\\
\item \quad si $\|F(x_0\|\le \eta$ alors \quad\quad\quad
  \textsf{test justifi\'e par le corollaire~\ref{cl_evaluation}}
 \item  \quad\quad $r:=\textbf{rang num\'erique}(DF(x_0))$ 
\item \quad \quad si $r<n$ alors
\item \quad  \quad \quad $F:=S(K(F))$
\item \quad  \quad\quad aller en $4$
\item \quad\quad sinon
\item \quad\quad \vbox{$\dfl(f)$ un syst\`eme d\'eflat\'e carr\'e de rang
    num\'erique $n$ extrait de $F$}
\item \quad\quad fin si
\item \quad fin si
\item \quad Sortie : $\dfl(f)$
\end{enumerate}
    \end{minipage}
 }
 $$
   \caption{}\label{dfl_table}
 \end{table}
 \begin{defin} 
Nous appelons l'op\'erateur de Newton classique associ\'e au syst\`eme d\'eflat\'e 
$dfl(f)$ de $\varepsilon$-rang $n$ l'op\'erateur de Newton
\emph{singulier} du syst\`eme initial $f$.
\end{defin}
Plut\^ot que de calculer la suite de d\'eflation introduite dans la
d\'efinition~\ref{dfl_seq}, il est suffisant de la tronquer. Pour ce
faire nous avons besoin de la d\'efinition suivante.
\begin{defin}\label{dfl_tr}
Soit $p\ge 1$. Nous notons $Tr_{x_0,p}(F)$ la s\'erie tronqu\'ee \`a l'ordre
$p$ de la fonction analytique $F$ au point $x_0$.\\
Nous appelons alors \emph{suite de d\'eflation tronqu\'ee} \`a l'ordre $p$
au point $x_0$ la suite~:
\begin{align*}
T_0&=Tr_{x_0,p}(S(f))
\\
T_{k+1}&=Tr_{x_0,p-k-1}\left (S\,(\,K(T_k)\,)\,
\right ),\,\quad 0\le k\le p.\end{align*} 
\end{defin}
Pour d\'efinir l'op\'erateur de Newton singulier il est alors suffisant de
conna\^\i tre l'\'epaisseur de la suite de d\'eflation.
 \begin{pp}
Soit $\ell$ l'\'epaisseur de la suite de d\'eflation. Consid\'erons la suite de d\'eflation tronqu\'ee
$(T_k)_{k\ge 0}$ \`a l'ordre $\ell +1$ au point $x_0$
(d\'efinition~\ref{dfl_tr}). Alors l'op\'erateur de Newton singulier
associ\'e \`a $f$ est \'egal \`a l'op\'erateur de Newton classique associ\'e \`a $T_{\ell}$. 
 \end{pp}
 \begin{proof}
 Comme $T_0$ est la s\'erie tronqu\'ee \`a l'ordre $\ell$ de $F_0$, par
 construction il est ais\'e de voir que pour tout $k=0:\ell$, $T_k$ est
 la s\'erie tronqu\'ee de $F_k$ \`a l'ordre $p-k$. Il s'ensuit la conclusion.
 \end{proof}
 \begin{table}
 $$\fbox{
 \begin{minipage}{1\textwidth }
$$ \textbf{Newton singulier}$$
\begin{enumerate}[1-]
\item \quad  Entr\'ees~: $x_0\in \C^n$, $f\in\A2(x_0,R_{x_0})^s$
\item \quad $    \dfl(f)=\textsf{syst\`eme d\'eflat\'e}(f)$
\item \quad Sortie~: Si $\dis \dfl(f)\ne \emptyset$ alors
  $N_{\dfl(f)}(x_0)$ sinon $x_0$\index{$N_{\dfl(f)}$}
\end{enumerate}
    \end{minipage}
 }
 $$
 \caption{}\label{Ndflx0}
 \end{table}
\section{ $\,$La multiplicit\'e chute strictement lors de la
    d\'eflation}
\label{Mult-drops-kerneling}
Dans cette section nous d\'emontrons que la suite de d\'eflation stationne
apr\`es un indice fini. Commen\c cons par le montrer dans le cas de la d\'eflation exacte. La proposition suivante montre que la multiplicit\'e chute strictement lors d'une op\'eration de d\'enoyautage.
\begin{tm}\label{drop_schur}
Supposons que le rang de $Df(\zeta)$ soit \'egal \`a $r$ et que 
$$Df(x):=\left (
\begin{array}{cc}
A(x)&B(x)\\C(x)&D(x)
\end{array}
\right )
$$ 
o\`u $A(\zeta)\in \C^{r\times r}$ est inversible. Alors la multiplicit\'e
de $\zeta$ comme racine de $K(f)$ est strictement plus petite que la
multiplicit\'e de $\zeta$ en tant que racine de $f$.
\end{tm}
\begin{proof}
Si $r=0$ alors le syst\`eme $K(f)$ est form\'e de toutes les d\'eriv\'ees
partielles $$\dis \frac{\partial f_i(x)}{\partial x_j },\quad 1\le j\le
n,\quad 1\le i\le s .$$
Alors la conclusion d\'ecoule du lemme~\ref{drop_K}.

Si $r>0$ le syst\`eme $K(f)$ est form\'e par $f_1,\ldots,f_r$ et les
\'el\'ements du compl\'ement de Schur $D(x)~-~C(x)~A(x)^{-1}~B(x)$.
De la proposition~\ref{schur}, les relations entre les lignes de la
matrice jacobienne sont $$(C(x),D(x))-C(x)A(x)^{-1}(A(x),B(x))=0.$$
Il est facile de voir que le syst\`eme $K(F)=0$ est analytiquement
\'equivalent en la racine $\zeta$ au syst\`eme suivant
  \begin{equation}\label{KF_eq}
  \left (f_1,\ldots, f_r, D f_{i}(x)-\sum_{j=1}^r\lambda_{ij}(x)D f_j(x) =0,\quad i=r+1:s\right )=0,
  \end{equation}
avec $( \lambda_{ij}(x))$ la $(s-r)\times r$-matrice $\left
  (C(x)A(x)^{-1} \right )$. 
\\ 
En appliquant le th\'eor\`eme des fonctions implicites, nous savons qu'il
existe un isomorphisme local $\Phi$ tel que
$$x_{1:r}-\zeta_{1:r}=f_{1:r}\circ \Phi.$$
En substituant $x_{1:r}-\zeta_{1:r}$ dans $f=0$ nous obtenons le syst\`eme
\begin{equation}\label{KF_eq1}
(x_1-\zeta_1,\ldots, x_r-\zeta_r,\, f_{r+1:s}\circ\Phi )=0.
\end{equation}
L'id\'eal engendr\'e par $f_{r+1:s}\circ\Phi$ contient seulement
les mon\^omes $x_i-\zeta_i$, $i=r+1:n$.
D'un autre c\^ot\'e remarquons que la multiplicit\'e de la racine $\zeta$
du syst\`eme~(\ref{KF_eq1}) n'a pas
chang\'e~: c'est aussi la multiplicit\'e de $\zeta_{r+1:n}$ comme racine
de $f_{r+1:s}\circ\Phi$. De plus, la multiplicit\'e de $\zeta$ comme
racine du syst\`eme~(\ref{KF_eq}) est \'egale \`a la multiplicit\'e de
$\zeta_{r+1:n}$ comme racine du syst\`eme $D(f_{r+1:s}\circ\Phi)$.
Nous appliquons maintenant le lemme~\ref{drop_K} au syst\`eme
$f_{r+1:s}\circ\Phi$ pour en d\'eduire que la multiplicit\'e chute.
\end{proof} 
\begin{pp}\label{schur}
Soit $\dis M=\left (\begin{array}{cc}
A&B \\C&D  
  \end{array}\right )\in \C^ {s\times n}$ of rang $r$ 
  o\`u $A\in\C^{r\times r}$ est inversible. Alors les relations entre
  les lignes (respectivement les colonnes) de $M$ sont donn\'ees par
  $$ (C,D)-CA^{-1}(A,B)=0,\quad
   (\textrm{respectivement $\dis  \left (\begin{array}{c}
B \\D  
  \end{array} \right )-\left (\begin{array}{c}
A \\C  
  \end{array} \right )A^{-1}B=0$} ).$$ 
\end{pp}
\begin{proof} La proposition vient de l'\'equivalence~: 
\\ $(C,D)-CA^{-1}(A,B)=0$ et $\dis \left (\begin{array}{c} B \\D
  \end{array} \right )-\left (\begin{array}{c}
A \\C  
  \end{array} \right )A^{-1}B=0$
  si et seulement si $D-CA^{-1}B=~0$. Comme le rang de la matrice $M$
  est \'egal \`a $r$, c'est \'equivalent \`a $\schur (M)~=~0$.
\end{proof}
\begin{defin}
La valuation d'un syst\`eme analytique $f=(f_1,\ldots,f_s)$ en $\zeta$
est le minimum des valuations des $f_i$ en $\zeta$.
\end{defin}
\begin{req}
Un g\'en\'erateur de $I\C\{x-\zeta\}$ de valuation minimale peut toujours
\^etre pris comme \'el\'ement d'une base standard (minimale) de $I$.
\end{req}

C'est une cons\'equence d'une propri\'et\'e fondamentale des ordres
locaux~: la valuation d'une somme est toujours plus grande que la
valuation de chacun de ses termes.\\

Dans le cas d'un localis\'e d'un syst\`eme polynomial, la
construction d'une base standard de $I\C\{x-\zeta\}$ \`a partir d'un
ensemble de g\'en\'erateurs peut \^etre r\'ealis\'ee par l'algorithme du c\^one
tangent de Mora, par calcul successif de $S$-polyn\^omes et des
r\'eductions qui en constituent un cas particulier. La valuation ne peut
que cro\^\i tre lors des ces op\'erations, ce qui emp\^eche de r\'eduire
$S(f,g)$ par $f$ (ou $g$).
\begin{lm}\label{drop_K} 
Soit $\dis D^1 f(x):=\left (\frac{\partial f_i(x)}{\partial x_j
  },\quad 1\le j\le n,\quad 1\le i\le s\right )$. Supposons que
$\zeta$ soit un z\'ero isol\'e de $f$ et $D^1 f$. Alors la multiplicit\'e
de $\zeta$ comme racine de $D^1 f$ est strictement plus petite que
la multiplicit\'e de $\zeta$ comme racine de $f$.
\end{lm}
\begin{proof}
Prenons un des $f_k$, disons $f_i$, de valuation minimale en
$\zeta$. Cette valuation est au moins $2$. C'est donc qu'il existe un
indice $j$ tel que le terme dominant $\dis \frac{\partial
  f_i(x)}{\partial x_j }$ n'est pas dans l'id\'eal engendr\'e par
$f$. D'ou la conclusion.
\end{proof}
L'op\'erateur de s\'election fait chuter la multiplicit\'e comme l'\'enonce la
proposition ci-dessous.
\begin{lm}\label{drop_dfl}
Soit $p$ la valuation de $f$ en $\zeta$. Consid\'erons le syst\`eme 
$$\dis
D^{p-1}f(x)~:~=~\left (\frac{\partial^{|\alpha|}f_i(x)}{\partial x^\alpha
},\, |\alpha|= p-1,\, 1\le i\le s\right ).$$ Supposons que $p\ge
2$ et que le rang de $D^pf(\zeta)$ soit \'egal \`a $r$.
Alors la multiplicit\'e de $\zeta$ comme racine de $ D^{p-1}f(x)=0$ est
strictement plus petite que la multiplicit\'e de $\zeta$ comme racine de
$f$. Plus pr\'ecis\'ement la multiplicit\'e de la racine $\zeta$ chute d'au
moins $p^r$.
\end{lm}
\begin{proof}
Comme la valuation est $p\ge 2$ alors $\dis f(x)=\sum_{k\ge
  p}\frac{1}{k!}D^kf(\zeta)(x-\zeta)^k$ avec $D^pf(\zeta)\ne 0$. Les
mon\^omes de $LT(f)$ sont de type $(x-\zeta)^\alpha$ avec $|\alpha|\ge
p\ge 2$. Donc le nombre de mon\^omes standards de $\dis
\C\{x-\zeta\}/LT(f)$ est plus grand que $p^n$. Comme le rang de la
d\'eriv\'ee de $D^{p-1}f(x)$ en $\zeta$ est $r>0$, nous pouvons supposer
sans perte de g\'en\'eralit\'e que $x_1-\zeta_1,\ldots, x_r-\zeta_r$ sont
dans l'id\'eal $LT(D^{p-1}f(x))$, et donc par cons\'equent le nombre de
mon\^omes standard chute d'au moins $p^r$.
\end{proof}
Lorsque la d\'eflation est effectu\'ee en un point $x_0$ suffisament 
proche de $\zeta$ l'\'evaluation en $f(x_0)$          sera petite. La section~\ref{evaluation_sec} quantifie la proximit\'e de $x_0$ \`a $\zeta$.
Le r\'esultat ci-dessous d\'etermine le rayon d'une boule centr\'ee en $\zeta$
dans laquelle le rang num\'erique de $Df(x_0)$ est identique \`a celui de $Df(\zeta)$.
\begin{pp}
Supposons que le rang de $Df(\zeta)$ est \'egal \`a $r$
et que ses valeurs singuli\`eres v\'erifient
$$\sigma_1(\zeta)\ge\ldots\ge\sigma_r(\zeta)>
\sigma_{r+1}(\zeta)=\ldots=\sigma_n(\zeta)=0.$$
Notons par $$\bar{\gamma}(f,\zeta)=\sup_{k\ge 2}\left (
\frac{\|D^kf(\zeta)\|}{k!}\right )^{\frac{1}{k-1}}.
\index{$\bar{\gamma}(f,\zeta )$}
$$
Soit $\dis 0\le \epsilon<
\min\left (2-\sqrt{2},\,\frac{\sigma_r(\zeta)}{2}\right).$
Pour tout $\dis x_0\in B\left (\zeta,\frac{\epsilon}{2\bar{\gamma}(f,\zeta)}\right )$
le $\varepsilon$-rang  de $Df(x_0)$ est \'egal au rang de $Df(\zeta)$.
\end{pp}
\begin{proof}
Notons par $\sigma_k(x_0)$ les valeurs singuli\`eres de $Df(x_0)$. Nous avons successivement pour  $k=1:n$ :
\begin{align*}
|\sigma_k(x_0)-\sigma_k(\zeta)|&\le 
\|   Df(x_0)-Df(\zeta)                 \|
\quad\quad \textrm{par le th\'eor\`eme de  Weyl~\cite{weyl1912} }
\\
&\le 
\sum_{k\ge 2}(k-1)\frac{\|D^kf(\zeta)\|}{k!}\|x_0-\zeta\|^{k-1}
\\
&< \sum_{k\ge 2}(k-1)\left (\frac{\epsilon}{2}\right )^{k-1}
\quad\quad \textrm{puisque  $\|x_0-\zeta\|< \frac{\epsilon}{2\bar{\gamma}(f,\zeta)}$ }
\\
&< \frac{\epsilon/2}{(1-\epsilon/2)^2}
\\
&< \epsilon\quad\quad \textrm{puisque $\epsilon< 2-\sqrt{2}$ }
\end{align*}
Puisque $\sigma_r(\zeta)> 2\epsilon$ nous en d\'eduisons que pour $k\le r$ 
$$\epsilon<\sigma_r(\zeta)-\epsilon< \sigma_k(\zeta)-\epsilon<\sigma_k(x_0).$$
D'autre part nous avons pour $k>r$  :
$$\sigma_k(x_0)<\sigma_k(\zeta)+\epsilon=\epsilon.$$
Finalement nous avons pour $k\le r$ et $j>r$
$$\sigma_j(x_0)<\epsilon<\sigma_k(x_0).$$
Il s'ensuit que le $\varepsilon$-rang de $Df(x_0)$ est \'egal \`a $r$.
\end{proof}
\section{ $\,$Un nouvel $\alpha$-th\'eor\`eme fond\'e sur le noyau de Bergman}\label{sec_rouche_regulier} 
Dans cette partie nous consid\'erons comme pr\'ec\'edemment $\omega\in \C^n$ et $f\in\A2(\omega, R_{\omega})^s$. Pour $x\in B(\omega,R_\omega)$
nous introduisons les quantit\'es
\begin{eqnarray}
\dis \beta(f,x)&=&\|Df(x)^{-1}f(x)\|
\label{eq_betafx}\index{$\beta(f,x)$}
\\
\dis  \lambda(f,x)&=&\frac{\|f\|}{(1-\nu_x^2)^\frac{n+1}{2}}
\label{eq_lambdafx}\index{$\lambda(f,x)$}
 \\
\dis\kappa_{x}&=&\max\left (\,1,\,\frac{(n+1)}{R_{\omega}(1-\nu_{x}^2)}\right ) \label{eq_kappafx}
\\
\dis\mu(f,x)&=&\|Df(x)^{-1}\|\label{eq_mufx}
\index{$\mu(f,x)$}
\\
\nonumber \\
\dis\gamma(f,x)&=&\max\left (\,1,\,\lambda(f,x)\,\kappa_x\, \mu(f,x)\,\right )  \label{eq_gammafx}\index{$\gamma(f,x)$}
\\
\nonumber \\
\dis \alpha(f,x)&=&\beta(f,x)\,\kappa_{x}\label{eq_alphafx}
\index{$\alpha(f,x)$}.
\end{eqnarray}
Nous pouvons remarquer que les quantit\'es  $\gamma(f,x)$ et $\alpha(f,x)$ sont diff\'erentes de celles introduites dans la $\alpha$-th\'eorie de Shub-Smale.
Ce parti pris d'utiliser les m\^emes notations  est justifi\'e respectivement par les th\'eor\`emes~\ref{rouche_regular} et~\ref{new_gamma_tm}.
D'une part la quantit\'e $\alpha$ du th\'eor\`eme~\ref{rouche_regular}
est relative \`a l'existence d'une racine comme dans le classique $\alpha$-th\'eor\`eme de~\cite{BCSS98} page 164 . D'autre part la quantit\'e $\gamma$ du th\'eor\`eme~\ref{new_gamma_tm}
est relative au rayon d'une boule de convergence quadratique de la m\'ethode de Newton  comme dans le classique $\gamma$-th\'eor\`eme de~\cite{BCSS98} page 156. 
Nous pouvons aussi ajouter que la reproduction des fonctions analytiques de carr\'e int\'egrables par le noyau de Bergman conduit naturellement \`a consid\'erer respectivement les quantit\'es  $\gamma(f,x)$ et $\alpha(f,x)$.
\begin{tm}\label{rouche_regular}($\alpha$-th\'eor\`eme).
Soient $R_{\omega}>0$, $x_0\in B(\omega, R_{\omega})$, 
et $f=(f_1,\ldots,f_n) \in
(\A2(\omega, R_{\omega}))^n$.
Nous notons $\alpha$, $\beta$, $\lambda$, $\gamma$, $\mu$, $\kappa$ pour
$\alpha(f,x_0)$, etc~$\ldots$ 
respectivement d\'efinis ci-dessus.
\\
Supposons que 
$$\alpha< 2\gamma+1-\sqrt{(2\gamma+1)^2-1}.$$
Alors pour tout $\theta>0$ tel que $B(x_0,\theta)\subset
B(\omega,R_{\omega})$ et
$$
\frac{ \alpha+1 -\sqrt{(\alpha+1)^2-4\alpha(\gamma+1)} }{2( \gamma +1)}< 
u:=\kappa\theta < \frac{ 1}{\gamma+1 }$$
    $f$ poss\`ede une unique racine dans la boule $B(x_0,\theta)$. 
  \end{tm}
Avant de prouver ce th\'eor\`eme nous aurons besoin de la proposition suivante~:
\begin{pp}\label{Berg_F_DFk_w}
Pour tout $f\in \A2(\zeta,R_{\omega})^s$ nous avons
 $$\dis\forall k\ge 0,\quad  \frac{1}{k!}\|D^k{f}(x_0)\|\le
||f||\frac{(n+1)^k }{R_\omega^k\left (1-\nu_{x_0}^2\right
  )^{\frac{n+1}{2}+k}}.$$
\end{pp}
\begin{proof}
Il suffit de tenir compte de l'in\'egalit\'e
$$ \frac{(n+1)\ldots (n+k)}{k!}\le (n+1)^k$$
  dans la proposition~\ref{Berg_F_DFK}.
\end{proof}
Nous pouvons maintenant prouver le th\'eor\`eme~\ref{rouche_regular}.
\begin{proof}
Nous consid\'erons $Df(x_0)^{-1}f(x)=Df(x_0)^{-1}f(x_0)+g(x)$
avec
 $$\dis g(x)=x-x_0+\sum_{k\ge 2}\frac{1}{k!}
Df(x_0)^{-1}D^kf(x_0)(x-x_0)^k.$$
Nous remarquons premi\`erement que pour tout $x\in\C^n$ nous avons 
  \begin{eqnarray}\label{rouche_gx}
  \|g(x)\| &\ge &
  \|x-x_0 \|-\sum_{k\ge 2}\frac{1}{k!}
\|Df(x_0)^{-1} D^kf(x_0)\|\, \|x-x_0 \|^k\nonumber
\\
&
\ge &  \|x-x_0 \|-
\frac{\|f|\| \,\|Df(x_0)^{-1}\| }{(1-\nu_{x_0}^2)^\frac{n+1}{2}} 
\sum_{k\ge 2} \left (
\frac{(n+1)\,\|x-x_0 \|}{R_{\omega}(1 -\nu_{x_0}^2)}
  \right )^k
\quad \textsf{de la proposition~\ref{Berg_F_DFk_w}}\nonumber
\\
&
\ge
&
\frac{u}{\kappa} - \frac{\gamma}{\kappa}\sum_{k\ge 2} u^k
\quad \textsf{de la d\'efinition de $\gamma=\lambda\kappa\mu$ et $u=\kappa\,\|x-x_0\|$}\nonumber\\
  &
  \ge
  &
\frac{1}{\kappa}\left (u-\gamma \frac{u^2}{1-u}\right ).
   \end{eqnarray}
  Soit $\theta>0$.
   Le th\'eor\`eme de Rouch\'e \'enonce que les applications  analytiques
$Df(x_0)^{-1}f(x)$
et $g(x)$
ont le m\^eme nombre de racines, chacune d'elles compt\'ees avec leurs multiplicit\'es respectives, dans la boule $B(x_0,\theta)$ si l'in\'egalit\'e
$$\|Df(x_0)^{-1}f(x)-g(x)\|< \|g(x)\|$$
est satisfaite pour tout $x\in \partial B(x_0,\theta)$.
Tout d'abord montrons que $x_0$ est l'unique racine de $g(x)$
dans la boule $\dis B\left (x_0,\frac{1}{ \kappa( \gamma +1)}\right )$. En effet consid\'erons $y$ une racine de $g(x)$ distincte de $x_0$ dans la boule
 $B(\omega,R_{\omega})$.
Nous posons $v=\kappa \|y-x_0\|$. Si $v\ge 1$ alors
$\dis \|y-x_0\|\ge 
1/\kappa>\frac{1}{\kappa(\gamma +1)}$. Par hypoth\`ese nous savons que  $\dis \frac{1}{\kappa(\gamma +1)}> \theta$. Donc dans le cas o\`u $v\ge 1$ nous concluons que $y\notin B(x_0,\theta)$.
Sinon  $v<1$.
Nous d\'eduisons de l'in\'egalit\'e~(\ref{rouche_gx}) que
   $$\|g(y)\|=0\ge \frac{1}{\kappa} \left (v -\frac{\gamma  v^2}{1- v}\right 
). $$
   Donc $\dis\frac{1}{\gamma +1}\le v$. 
   Il s'ensuit que la distance entre les deux racines $x_{0}$ et $y$ de $g(x)$ est minor\'ee par
 $$\dis \|y-x_0\|\ge\frac{1}{\kappa (\gamma+1 )}>\theta. $$
 Ceci montre que $x_0$ est la seule racine de  $g(x)$
 dans la boule  $\dis  B\left (x_0,  \frac{1 }{    \kappa (\gamma+1 )} \right)$.
   \\
Maintenant nous consid\'erons $x\in B(\omega,R_{\omega})$ tel que $\dis \|x-x_0 
\|=\theta=\frac{u}{\kappa}$.
Alors $B(x_0,\theta)\subset B(\omega, 
R_{\omega})$. De l'in\'egalit\'e~(\ref{rouche_gx}) nous d\'eduisons que l'in\'egalit\'e
\begin{equation}\label{eq_rouche}
\beta:=\|Df(x_0)^{-1}f(x_0)\|< \frac{ 1 }{\kappa } \left (u 
-\frac{\gamma u^2}{1- u}\right )
\end{equation}
implique $\|Df(x_0)^{-1}f(x)-g(x)\|< \|g(x)\|$ sur la fronti\`ere de la boule $B(x_0,\theta)$.
Puisque $\alpha=\beta\kappa$, ceci est satisfait si le num\'erateur
$$(\gamma+1)u^2-(\alpha+1)u+\alpha$$
de l'expression pr\'ec\'edente~(\ref{eq_rouche}) est strictement n\'egative. Alors il est facile  de voir que sous la condition
$$\alpha:=\beta\kappa< 2\gamma+1-\sqrt{(2\gamma+1)^2-1}$$
le trin\^ome $(\gamma+1)u^2-(\alpha+1)u+\alpha$
poss\`ede deux racines \'egales \`a
\\ $\dis \frac{ \alpha+1 \pm\sqrt{(\alpha+1)^2-4\alpha(\gamma+1)} }{2( 
\gamma+1)}$. 
Donc pour tout $\theta$ tel que
$$ \frac{ \alpha+1 -\sqrt{(\alpha+1)^2-4\alpha(\gamma+1)} }{2( 
\gamma+1)}<  u:=\kappa\theta<\frac{ 1}{\gamma+1 }$$
nous avons $(\gamma+1)u^2-(\alpha+1)u+\alpha<0$.  Alors l'in\'egalit\'e~(\ref{eq_rouche}) est satisfaite et le syst\`eme
$f$ poss\`ede une unique racine dans la boule $ \dis B(x_0,\theta)$. Le th\'eor\`eme est d\'emontr\'e.
\end{proof}
\section{Un nouveau $\gamma$-th\'eor\`eme
fond\'e sur le noyau de Bergman		}\label{sec_NS_gamma}
Soit $f=(f_1,\ldots , f_n)$ un syst\`eme analytique r\'egulier en une de ces racines $\zeta$. Le rayon de la boule dans lequel la suite de Newton converge quadratiquement est contr\^ol\'e par la quantit\'e 
  $$\gamma(f,\zeta)=\sup_{k\ge 2}\left (
\frac{1}{k!}\|Df(\zeta)^{-1}D^kf(\zeta)\| \right )^{\frac{1}{k-1}}$$
introduite par M. Shub et S. Smale. Plus pr\'ecis\'ement nous avons le r\'esultat suivant appel\'e $\gamma$-th\'eor\`eme.
  \begin{tm} ($\gamma$-theorem de~\cite{BCSS98}). Soit
   $f(x)$ un syst\`eme analytique et  $\zeta$ une racine r\'eguli\`ere de $f(x)$. Soit $\dis
    R_\zeta=\frac{3-\sqrt{7}}{2\gamma(f,\zeta)}$. Alors pour tout $x_0\in 
B(\zeta,R_{\zeta})$
    la suite de Newton    $$x_{k+1}=x_k-Df(x_k)^{-1}f(x_k),\quad k\ge 0,$$
     converge quadratiquement vers $\zeta$.
  \end{tm}
Nous donnons ici une version d'un $\gamma$-th\'eor\`eme qui prend en compte le noyau de  Bergman pour reproduire les fonctions analytiques de carr\'e int\'egrables.
\begin{tm}\label{new_gamma_tm}($\gamma$-th\'eor\`eme).
Soit $\zeta$ une racine r\'eguli\`ere d'un syst\`eme analytique
$f=(f_1,\ldots,f_n)\in \A2( 
\omega ,R_{\omega})^n$.
Nous notons $\gamma$ et $\kappa$ pour $\gamma(f,\zeta)$ et $\kappa_\zeta$
respectivement  d\'efinis en~(\ref{eq_gammafx}) et~(\ref{eq_kappafx}).  
Alors pour tout  $x$ tel que
$$   \dis u:=\kappa \,\|x-\zeta\|
< \frac{2\gamma+1-\sqrt{ 4\gamma^2+3\gamma}}{\gamma+1}$$ 
la suite de  
Newton 
$$x_0=x,\quad x_{k+1}=N_f(x_k),\quad k\ge 0,$$
   converge quadratiquement vers $\zeta$. Plus pr\'ecis\'ement
$$\|x_k-\zeta|\le \left ( \frac{1}{2}\right )^{2^k-1}\,\|x-\zeta\|,\quad k\ge 0.$$
\end{tm}
\begin{proof} Nous utilisons la proposition~\ref{new_gamma_pp} ci-dessous
pour montrer par r\'ecurrence le r\'esultat. Le sch\'ema de la preuve est classique et peut \^etre trouv\'e dans in~\cite{BCSS98} page 158.
 L'hypoth\`ese $\dis   u<  \frac{2\gamma+1-\sqrt{ 
4\gamma^2+3\gamma}}{\gamma+1}$  implique que
$\dis   \frac{\gamma u}{(1+\gamma)(1-u)^2-\gamma}\le~\frac{1}{2 }$. C'est une condition suffisante  pour la convergence quadratique
de la suite de Newton 
avec une raison  de $\dis \frac{1}{2}$.
\end{proof}
\begin{pp}\label{new_gamma_pp}
Avec les notations du  th\'eor\`eme~\ref{new_gamma_tm} nous avons :
\begin{itemize}
\item[1--]   Pour tout $x$ satisfaisant $\dis u<
1-\sqrt{\frac{\gamma}{1+\gamma }}$,  $Df(x)$ est 
inversible. De plus nous avons 
$$\dis \|Df(x)^{-1}Df(\zeta)\|\le \frac{(1-u)^2}{(1+\gamma)\, 
(1-u)^2-\gamma};$$
\\
\item[2--] $\dis\| Df(\zeta)^{-1}\left (Df(x)(x-\zeta)-f(x)\right )\|\le
\frac{\gamma u^2}{(1-u)^2}
$;
\\
\\
\item[3--] $\dis \|N_f(x)-\zeta\|\le \frac{\gamma u^2}{ (1+\gamma)\, 
(1-u)^2-\gamma}.$
\end{itemize}
\end{pp}
\begin{proof}\quad
\\
{\hspace{-5cm}\begin{itemize}
\item[1--]
Nous \'ecrivons
  $$\dis Df(\zeta)^{-1}Df(x)-I=\sum_{k\ge 1}
\binomial{k+1}{k}Df(\zeta)^{-1}\frac{D^{k+1}f(\zeta)}{(k+1)!}(x-\zeta)^k.$$
De la   proposition~\ref{Berg_F_DFk_w}, il vient
 \begin{align*}
  \dis \frac{1}{(k+1)!}\|D^{k+1}f(\zeta)\|\,\|Df(\zeta)^{-1}\|
  &\le \frac{||f||\,\|Df(\zeta)^{-1}\|\,(n+1)^{k+1}}{R_{\omega}^{k+1}\left (1-\nu_{\zeta}^2\right )^{\frac{n+1}{2}+k+1}}
  \\\\
&\le \lambda\mu\kappa^{k+1}=\gamma\kappa^k.
\end{align*}
D'o\`u
  \begin{align*}
  \|Df(\zeta)^{-1}Df(x)-I\|&\le \gamma \sum_{k\ge 1}
\binomial{k+1}{k}\, \left (\kappa\,||x-\zeta\| \right )^k
\\
&\le
\gamma\left (\frac{1}{(1-u)^2}-1\right )
\end{align*}
avec $\dis u=\kappa\|x-\zeta\| .$
 Alors gr\^ace au lemme de Von Neumann, voir par exemple \cite{Kato13} page 30,
 l'assertion 1 suit facilement.
\item[2--] Nous avons
$\dis Df(x)(x-\zeta)-f(x)=\sum_{k\ge 2}(k-1)\frac{1}{k!}
D^kf(\zeta)(x-\zeta)^k$. Donc, utilisant de nouveau la
proposition~\ref{Berg_F_DFk_w}  
un calcul direct conduit \`a
\begin{align*}
\| Df(\zeta)^{-1}\left (Df(x)(x-\zeta)-f(x)\right )\|
&\le \gamma
\sum_{k\ge 2}(k-1)\left (\kappa\|x-\zeta\|\right )^k
\\
&\le
\frac{\gamma u^2}{(1-u)^2}.
\end{align*}
Ceci prouve l'assertion 2.
\item[3--] Nous avons
$$N_f(x)-\zeta=Df(x)^{-1}Df(\zeta)\, Df(\zeta)^{-1}(Df(x)(x-\zeta)-f(x)).$$
Des  items 1 et 2, nous d\'eduisons le r\'esultat.
\end{itemize}
  }
\end{proof}
\section{Estimation de la quantit\'e $\gamma$ du syst\`eme d\'eflat\'e}\label{sec_estim_gamma}\label{sec-gammal}
Nous consid\'erons les notations introduites pr\'ec\'edemment 
o\`u $\zeta\in B(\omega,R_\omega)$ est une racine du syst\`eme $F\in \A2(\omega,R_\omega)^s$.
Nous notons $\dis [F]_\zeta=\sum_{k\ge 0 }\frac{1}{k!}\|D^kF(\zeta)\|\,\|x-\zeta\|^k$.
\begin{lm}\label{estim_F_Berg}
Soient $\dis \kappa:=\kappa_\zeta$,
$\dis \lambda=\frac{\|F\|}{(1-\nu_\zeta^2)^\frac{n+1}{2}}$ et $F\in \A2(w,R_\omega)^s$
tel que
 $\dis F(x)= \sum_{k\ge p }\frac{1}{k!}D^kF(\zeta)(x-\zeta)^k$
 avec  $p\ge 1$. Nous notons
  $u=\kappa \|x-\zeta\|$.  Alors
$$[F]_\zeta\le \frac{\lambda u^p}{1- u }.$$
\end{lm}
\begin{proof}
De la majoration $\dis \frac{1}{k!}\|D^k{F}(\zeta)\|\le
||F||\frac{(n+1)^k }{R_\omega^k\left (1-\nu_{\zeta}^2\right )^{\frac{n+1}{2}+k}}=\lambda \kappa^k	$ donn\'ee par la 
proposition~\ref{Berg_F_DFk_w}, 
nous avons successivement :
\begin{align*}
[F]_\zeta & \le \lambda u^p\sum_{k\ge 0}u^k
\\
& \le 
\frac{\lambda u^p}{1-u}.
\end{align*}
\end{proof} 
\begin{lm}\label{pt}
 Soient $t\in [0,1[$ et $p\ge 1$.
$$\sum_{k\ge 0 }\binomial{p-1+k}{k}t^k=\frac{1}{(1-t)^p} .$$
\end{lm}
 \begin{proof}
Par r\'ecurrence. C'est vrai pour $p=1$. Supposons-le au cran $p$. Alors
 $$\sum_{k\ge 1 }k\binomial{p-1+k}{k}t^{k-1}=
 \frac{p}{(1-t)^{p+1}}$$
 et
 $$\sum_{k\ge 0 }\frac {k+1}{p }\binomial{p+k}{k+1}t^{k}=
 \frac{1}{(1-t)^{p+1}}.$$
 Donc
  $$\sum_{k\ge 0 }\binomial{p+k}{k}t^{k}=
 \frac{1}{(1-t)^{p+1}}.$$
 \end{proof} 
 \begin{lm}\label{mt}
Pour tout $p \ge 1$ et $u\in[0,2/(p +1)[$ nous avons
$$\frac{1}{(1-u)^{p }}-1\le \frac{p \,u}{1-\frac{p +1}{2}u}.$$
\end{lm}
\begin{proof}
L'in\'egalit\'e est vraie pour 	$p =1$. Supposons-la pour $p$ donn\'e. Puisque
$(1-p u)(1-u)=1-(p +1)u+p u^2\ge 1-(p +1)u$ nous avons successivement :
\begin{align*}
\frac{1}{(1-u)^{p +1}}-1&\le \left( \frac{p u}{1-\frac{p +1}{2}u}+1 \right )\frac{1}{1-u}-1\\
&\le \frac{(p +1)u(1-u/2)}{(1-\frac{p +1}{2}u)(1-u)}
 \end{align*}
De plus, pour $u\in[0,2/(p +1)]$ nous avons :
\begin{align*}
\frac{1}{1-\frac{p +2}{2}u}-\frac{1-u/2}{(1-\frac{p +1}{2}u)(1-u)}&=
\frac{p u^2}{4(1-\frac{p +1}{2}u)(1-\frac{p +2}{2}u)(1-u)}  \ge 0.
\end{align*} 
Il s'ensuit :
\begin{align*}
\frac{1}{(1-u)^{p +1}}-1 
&\le \frac{(p +1)u}{1-\frac{p +2}{2}u}.
 \end{align*}
 L'in\'egalit\'e est vraie au cran $p+1$. Le lemme est d\'emontr\'e.
 \end{proof}
\begin{lm}\label{estim_DFp_Berg} 
Soit $p\ge 1.$ Avec les hypoth\`eses du lemme~\ref{estim_F_Berg} nous avons~:
$$ \left [\frac{1}{(p-1)!}\left (
	D^{p-1}F 
	-
	  D^{p-1}F(\zeta)\right ) \right ]_\zeta
\le
\dis  \lambda\kappa^{p-1} \left (\frac{1}{\left (
1- u \right )^{p}}-1\right )\le \lambda\kappa^{p-1}\frac{p \,u}{1-\frac{p +1}{2}u}.$$ 
\end{lm}
\begin{proof}
En proc\'edant comme dans la preuve du lemme~\ref{estim_F_Berg} nous
obtenons successivement :
\begin{align*} 
	\left [\frac{1}{p-1)!}\left (
	D^{p-1}F\right . \right .&
	-
	\left .\left . D^{p-1}F(\zeta)\right ) \right ]_\zeta
\le \sum_{k\ge 1}\frac{(p-1+k)!}{(p-1)!\, k! }
\frac{
\|D^{p-1+k}F(\zeta)||}{(p-1+k)!}\|x-\zeta\|^k
\\
&\le\lambda\kappa^{p-1}\sum_{k\ge 1}\binomial{p-1+k}{k} u^{k}
\\
&\le   \lambda\kappa^{p-1} \left (\frac{1}{\left (
1- u \right )^{p}}-1\right ) \textsf{\qquad(par le lemme~\ref{pt})}
\\
&\le   \lambda\kappa^{p-1} \frac{p \,u}{1-\frac{p +1}{2}u}
\textsf{\qquad(par le lemme~\ref{mt})}.
\end{align*}
\end{proof}
\begin{lm}\label{estim_inverse_DF_Berg}
Soient $r>0$ et $F=(F_{1:r},F_{r+1:s})\in \A2(\omega,R_\omega)^s$ tels que  
$DF(\zeta)$ soit de rang $r$ et $D_{1:r}F_{1:r }(\zeta)$ soit inversible.
\\
Nous notons $\dis \kappa:=\kappa_\zeta$,
$\dis \lambda=\frac{\|F\|}{(1-\nu_\zeta^2)^\frac{n+1}{2}}$
, $\mu=\|D_{1:r}F_{1:r}^{-1}(\zeta)\|$ et
$\gamma=\lambda\kappa\mu$. Nous introduisons aussi
$u=\kappa \|x- \zeta\|$. Pour tout $x\in B(\omega,R_\omega)$
tel que $\dis u<1-\sqrt{\frac{\gamma}{1+\gamma}}$
il s'ensuit que $D_{1:r}F_{1:r}(x)$ est inversible.
De plus nous avons la majoration :
 $$[D_{1:r}F_{1:r}^{-1} -D_{1:r}F_{1:r}(\zeta)^{-1}]_\zeta \le \frac{ \gamma\mu\,v\,u}{1-\gamma\, v\,u }$$
 o\`u $ \dis v=\frac{2-u}{(1-u)^2}.$
\end{lm}
\begin{proof} 
Nous avons $\dis D_{1:r}F_{1:r}(x)=D_{1:r}F_{1:r}(\zeta)+\sum_{k\ge 1}\frac{1}{k!}D^k\left ((D_{1:r}F_{1:r}\right )(\zeta)(x -\zeta)^k$.
Puisque $D_{1:r}F_{1:r}(\zeta)$ est inversible et que 
$\dis \|D^k\left ((D_{1:r}F_{1:r}\right )(\zeta)\|\le 
\|D_{1:r}^{k+1}F_{1:r}(\zeta)\|$ nous pouvons \'ecrire en proc\'edant
comme dans la preuve du lemme~\ref{estim_DFp_Berg} :
\begin{align*}
||E\|:=\|D_{1:r}F_{1:r}(\zeta)^{-1}D_{1:r}F_{1:r}(x)-I||& \le  \sum_{k\ge 1}\frac{1}{k!}\|D^k\left ((D_{1:r}F_{1:r}\right )(\zeta)
\|\,\|x -\zeta)^k\|
\\
&\le  \mu
 \sum_{k\ge 1}\binomial{k+1}{k}
\frac{\|\,D^{k+1}F(w)||}{(k+1)!}|x-\zeta\|^k
\\
&\le \lambda\kappa\mu 
\sum_{k\ge 1} \binomial{k+1}{k} u^k
\\
&\le \gamma\left(\frac{1}{\left (1-u\right )^2}-1\right )
\\
&\le \frac{ \gamma (2-u)u}{(1-u)^2}=\gamma\, v\,u.
\end{align*}
La condition $u<1-\sqrt{\frac{\gamma}{1+\gamma}}$ implique $\|E\|\le  \gamma v<1$.
Donc $D_{1:r}F_{1:r}(x)$ est inversible.
De plus
 $D_{1:r}F_{1:r}(x)^{-1}=(I+E)^{-1}D_{1:r}F_{1:r}(\zeta)^{-1}.$
Alors nous avons
\begin{align*}
D_{1:r}F_{1:r}(x)^{-1} -D_{1:r}F_{1:r}^{-1}(\zeta)&=\left (\sum_{k\ge 1 } E^k\right )D_{1:r}F_{1:r}( \zeta)^{-1}. 
\end{align*}
Finalement
\begin{align*}
\|D_{1:r}F_{1:r}^{-1} -D_{1:r}F_{1:r}^{-1}(\zeta)\|&\le 
 \frac{\|D_{1:r}F_{1:r}(\zeta)^{-1}\|\,\|E\|}{1- \|E\|}
 \\
 &\le  \frac{ \gamma\mu\,v\,u}{1-\gamma\,v\,u} . 
\end{align*}
\end{proof}
\begin{pp}\label{estim_kerneling} 
 Soient  $\dis \kappa=\kappa_\zeta$,
$\dis \lambda=\frac{\|F\|}{(1-\nu_\zeta^2)^\frac{n+1}{2}}$
et $u=\kappa\|x-\zeta\|$. Soit $r$ le rang de $DF( \zeta)$. 
Nous
supposons que $r>0$ et que
$D_{1:r}F_{1:r}(\zeta)$
est inversible. Nous notons
$\mu=\|D_{1:r}F_{1:r}(\zeta)^{-1}\|$ et $ \gamma=\lambda\kappa\mu$.
Alors nous avons :
 \begin{align}\label{eq_KF}
 [K(F)]_\zeta\le \frac{ \lambda u}{1-u}+ \frac{\lambda\kappa\,(1+\gamma)^2\, v\, u}{1-\gamma\, v\, u}
 \end{align}
 o\`u $ \dis v=\frac{2-u}{(1-u)^2}.$
 \end{pp}
 \begin{proof} 
  Nous pouvons \'ecrire :
$$DF(x)=\left (\begin{array}{cc}
D_{1:r}F_{1:r}(x)&D_{r+1:n}F_{1:r}(x)\\
D_{1:r}F_{r+1:m}(x)&D_{r+1:n}F_{r+1:m}(x)
\end{array}\right ):=\left   ( \begin{array}{cc}
A&B\\C&D
\end{array}\right ). $$
Nous avons $(D-CA^{-1}B)(\zeta)=0$.
Un cran de d\'eflation conduit \`a
    $K(F)=(F_{1:r},\vect (D-CA^{-1}B))$.
    Nous avons :
\begin{align*}
D-CA^{-1}B&=
 D-D(\zeta)+ (C(\zeta)-C)A^{-1}(\zeta)B(\zeta)
\\
&\quad + C(A^{-1 }(\zeta)-A^{-1})B(\zeta) + CA^{-1}(B(\zeta)-B).\quad  
(\textsf{puisque $(D-CA^{-1}B)(\zeta)=0$.)}
\end{align*}
Il s'ensuit :
\begin{align*}
 [D-CA^{-1}B]_\zeta&\le[D-D(\zeta)	]_\zeta
 \\
&\quad + [C-C(\zeta)]_\zeta \|A^{-1}(\zeta)\|\,\|B(\zeta)\|
\\
&\quad + [C]_\zeta\, [A^{-1 }-A^{-1}(\zeta)]_\zeta\,\|B(\zeta)\|
\\
&\quad   + [C]_\zeta[A^{-1}]_\zeta\,[B(\zeta)-B]_\zeta.
\end{align*}	
Rappelons les notations : $\dis \lambda=\frac{\|F\|}{(1-\nu^2_\zeta)^\frac{n+1}{2}}$, $\dis \kappa=\max\left (1,\,\frac{n+1}{R_\omega(1-\nu^2_\zeta)}\right ) $,
$\mu=\|A^{-1}(\zeta)\|$ et  $\dis v=\frac{2-u}{(1-u)^2}.$
\\
Nous avons les estimations successives :
\begin{align*}
 & 
 [D-D(\zeta)]_\zeta\le [DF-DF(\zeta)]_\zeta\le \lambda\kappa \, v\, u,
 \quad \textsf{du lemme ~\ref{estim_DFp_Berg} avec $p=2$,}
\\\\
&[A^{-1}(\zeta)(B-B(\zeta))]_\zeta,\quad[(C-C(\zeta))A^{-1}(\zeta)]_\zeta \le 
\lambda\kappa\mu\, v\, u=\gamma\,v\, u,
 \quad \textsf{du lemme ~\ref{estim_DFp_Berg} avec $p=2$}
\\\\
& \|B(\zeta)\|\le \|DF(\zeta)\|\le \lambda\kappa,
 \quad \textsf{de la proposition ~\ref{Berg_F_DFK} avec $k=1$}
\\\\
&[C]_\zeta\le ||DF(\zeta)||+[DF-DF(\zeta)]_\zeta\le \lambda\kappa
(1 + v\, u) ,
\\\\ 
& [A^{-1}-A(\zeta)^{-1}]_\zeta\le 
\frac{\gamma\mu\, v\, u}{1-\gamma\, v\, u} ,
 \quad \textsf{du lemme ~\ref{estim_inverse_DF_Berg}}
\\\\
&[A^{-1}]_\zeta\le \|A(\zeta)^{-1}\|+ [A^{-1}-A(\zeta)^{-1}]_\zeta
\le \mu + \frac{\gamma\mu\,v\, u}{1-\gamma\, v\, u}= \frac{\mu}{1-\gamma\, v\, u} . 
\end{align*}
Alors nous obtenons en tenant compte de ces majorations :
\begin{align*}
[D-CA^{-1}B]_\zeta
&\le  
 \lambda\kappa\, v\, u 
 +
 \lambda^2\, \kappa^2\, \mu v \, u
 +
\lambda\kappa\,(1+v\, u)
\frac{\gamma^2\, v\, u}{1-\gamma\, v\, u}
\\
&\quad
+
\lambda\kappa\,(1+v\, u)
\frac{\mu}{1-\gamma\, v\, u}
\lambda\kappa\,v\, u
\\
&\le  \frac{(1+\gamma)^2\lambda\kappa\,v\, u}{1-\gamma\, v\, u}.	
 \end{align*} 
D'un autre c\^ot\'e, le lemme~\ref{estim_F_Berg} avec $p=1$
implique
 \begin{align*}
 [F_{1:r}]_\zeta&\le \frac{ \lambda u}{1-u}. 
 \end{align*}
Nous en concluons que
 \begin{align*}
 [K(F)]_\zeta
 &\le 
 \frac{ \lambda u}{1-u}+
  \frac{\lambda\kappa\,(1+\gamma)^2\, v\, u}{1-\gamma\, v\, u}.
  \end{align*}
\end{proof}
\begin{pp}\label{estim_selection} 
 Soient  $\dis \kappa=\kappa_\zeta$,
$\dis \lambda=\frac{\|F\|}{(1-\nu_\zeta^2)^\frac{n+1}{2}}$
et $u=\kappa\|x-\zeta\|$. Soit $p$ la valuation de $F$ en $\zeta$.
Nous avons :
 $$[S(F)]_\zeta\le
 \lambda\kappa^{p-1}\frac{p \,u}{1-\frac{p +1}{2}u} .
 $$
 Si $\dis u\le \frac{1}{p+1}$ alors 
 \begin{align}\label{eq_SF}
  [S(F)]_\zeta\le
 \lambda\kappa^{p-1}\frac{2p}{p+1}\le 2\lambda\kappa^{p-1}.
 \end{align}
 \end{pp}
 \begin{proof}
Par construction de $S(F)$, c'est une cons\'equence directe du lemme~\ref{estim_DFp_Berg}.
 \end{proof}
Dor\'enavant nous supposerons
que chaque \'el\'ement $F_k$ de la suite de d\'eflation
\begin{align*}
F_0&=S(f)
\\
F_{k+1}&=S(\,K(F_k)\,),\quad k \ge 0.
\end{align*}
est de rang $r_k$. On sait par la  remarque~\ref{req_rang_dfl} que $r_k>0$.
Sans perte de g\'en\'eralit\'e nous pouvons dire que $D_{1:r_k}F_{k,1:r_k}$
est inversible.
\begin{tm}\label{estim_gammal}
Soient $f=(f_1,\ldots ,f_s)\in\A2(\omega,R_\omega)^s$ et $\zeta$ une racine de $f$.
Nous consid\'erons la suite de d\'eflation d'\'epaisseur $\ell$ de la d\'efinition~\ref{dfl_seq} :
 \begin{align*}
F_0&=S(f)
\\
F_{k+1}&=S(\,K(F_k)\,),\quad k \ge 0.
\end{align*}
Nous notons $p_0$ (respectivement, $p_k$) le maximum des valuations des \'equations du syst\`eme 
$f$ (respectivement, $K(F_k)$, $k=0:\ell-1$). Nous consid\'erons $p=\underset{k=0:\ell-1}{\max}p_k$.	
Soient  $\dis \kappa=\kappa_\zeta$ et
$\dis \lambda_k=\frac{\|F_k\|}{(1-\nu_\zeta^2)^\frac{n+1}{2}}$.
Soit $r_k$ le rang de $DF_k(\zeta)$ et $\mu =   \underset{k=0:\ell}{\max}||D_{1:r_k }F_{k,1:r_k}(\zeta)^{-1}||$.
Nous notons $\gamma_0=\frac{2p_0}{p_0+1}\lambda_0\kappa^{p_0}\mu $ et
$\gamma_k=\gamma(F_k,\zeta)$ pour $k\ge 1$. Nous consid\'erons $R>0$ tel que
$$\dis u:=\kappa R\le\min\left (\frac{1}{p+1},\quad
\frac{(1-\nu_\zeta^2)^{\frac{n+1}{2}} }{6(\ell +\gamma_0)\,\left (4\kappa^p(1+\ell+\gamma_0)+(1-\nu_\zeta^2)^{\frac{n+1}{2}}\right )}\,\right) .$$ Alors 
  la majoration 
\begin{equation}\label{gammal}
\gamma_\ell \le \ell + \gamma_0
\end{equation} 
 est vraie dans la boule $B(\zeta, R)$.
\end{tm}
\begin{proof}
La proposition~\ref{estim_selection} implique
$$\gamma(F_0,\zeta):=\lambda(F_0,\zeta)\kappa\mu\le \frac{2p_0}{p_0+1}\lambda_0\kappa^{p_0-1}\kappa\mu:=\gamma_0.$$ 
De la remarque~\ref{req_rang_dfl}
nous savons que $r_k\ge 1$.  
Nous proc\'edons par r\'ecurrence sur $k$ pour d\'emontrer l'in\'egalit\'e~(\ref{gammal}) qui est trivialement vraie pour $k=0$.
Supposons $\gamma_k\le  k+\gamma_0$ et montrons que $\gamma_{k+1}
\le 1+k+\gamma_0$.
Nous utilisons simultan\'ement les propositions~\ref{estim_kerneling} et ~\ref{estim_selection} pour \'ecrire
\begin{align}
\gamma_{k+1}
&\le \lambda_{k+1}\kappa\mu \nonumber		
\\
&\le
\frac{||F_{k+1}||}{(1-\nu_\zeta^2)^{\frac{n+1}{2}}}\kappa\mu \nonumber
\\&\le\frac{||S(K(F_{k}))||}{(1-\nu_\zeta^2)^{\frac{n+1}{2}}}\kappa\mu
\nonumber
\\
&\le \left ( \frac{\lambda_ku}{1-u}+\frac{\lambda_k\kappa(1+\gamma_k)^2vu	}{1-\gamma_kvu}\right  )\frac{2\kappa^p\,\mu}{(1-\nu_\zeta^2)^{\frac{n+1}{2}}}
 \qquad \textsf{des in\'egalit\'es~(\ref{eq_KF}), ~(\ref{eq_SF}) et $\dis v=\frac{2-u}{(1-u)^2}$		
  }\nonumber
\\
&\le  \left ( \frac{1}{1-u}+\frac{\kappa(1+\gamma_k)^2v}{1-\gamma_kvu}\right  )\frac{2\kappa^{p-1}\gamma_k u}{(1-\nu_\zeta^2)^{\frac{n+1}{2}}}
\qquad \textsf{puisque $\gamma_k=\max(1,\lambda_k\kappa\mu)$
  }\nonumber
  \\
&\le  \left ( \frac{1}{1-u}+\frac{\kappa(1+k+\gamma_0)^2v}{1-(k+\gamma_0)vu}\right  )\frac{2\kappa^{p-1}(k+\gamma_0) u}{(1-\nu_\zeta^2)^{\frac{n+1}{2}}}
\qquad \textsf{de l'hypoth\`ese de r\'ecurrence $\gamma_k\le k +\gamma_0$
  }\nonumber
\\
&\le 
  \left ( \frac{1}{1-u}+\frac{6\kappa(1+k + \gamma_0)^2}{(1-6(k+\gamma_0)u)}\right  )\frac{2\kappa^{p-1}(k+\gamma_0) u}{(1-\nu_\zeta^2)^{\frac{n+1}{2}}} \qquad \textsf{car $\dis u\le \frac{1}{p+1}\le \frac{1}{2}$ implique
  $\dis v\le 6$ }\label{gammak1}
  \\
  &\underset{\textrm{def}}{\le} U+V.\nonumber
  \end{align}
Le reste de la preuve consiste \`a montrer que les termes $U$ et $V$ sont plus petits que $(1+k+\gamma_0)/2$. Ainsi nous aurons $\gamma_{k+1}\le
1+k+\gamma_0$.
Nous avons
  \begin{align*}
  u\le \frac{(1+k+\gamma_0)(1-\nu_\zeta^2)^{\frac{n+1}{2}}}{(1+k+\gamma_0)(1-\nu_\zeta^2)^{\frac{n+1}{2}}+4\kappa^{p-1}(k+\gamma_0)}
\quad   \textrm{
implique
$\dis
U:= \frac{1}{1-u}\frac{2\kappa^{p-1}(k+\gamma_0) u}{(1-\nu_\zeta^2)^{\frac{n+1}{2}}}
 \le \frac{1}{2}(1+k+\gamma_0).
$
}
\end{align*}
Puisque $\kappa,\gamma_0,k\ge 1$, nous montrons par des majorations \'el\'ementaires  que :
$$
\frac{(1-\nu_\zeta^2)^{\frac{n+1}{2}} }{6(k +\gamma_0)\,\left (4\kappa^p(1+k+\gamma_0)+(1-\nu_\zeta^2)^{\frac{n+1}{2}}\right )}\le \frac{(1+k+\gamma_0)(1-\nu_\zeta^2)^{\frac{n+1}{2}}}{(1+k+\gamma_0)(1-\nu_\zeta^2)^{\frac{n+1}{2}}+4\kappa^{p-1}(k	+\gamma_0)}.
$$
Il s'ensuit que pour $\ell\ge k$ nous avons :
\begin{align}
  u\le  \frac{(1-\nu_\zeta^2)^{\frac{n+1}{2}} }{6(\ell +\gamma_0)\,
  \left (4\kappa^p(1+\ell+\gamma_0)+(1-\nu_\zeta^2)^{\frac{n+1}{2}}\right )}
\quad   \textrm{
implique
$\dis
 U:=\frac{1}{1-u}\frac{2\kappa^{p-1}(k+\gamma_0) u}{(1-\nu_\zeta^2)^{\frac{n+1}{2}}}
 \le \frac{1}{2}(1+k+\gamma_0).
$
}\label{ineq1_gammak1}
\end{align}
D'autre part un calcul direct \'etablit que :
\begin{align}
\hspace{-1.5cm}
u\le \frac{(1-\nu_\zeta^2)^{\frac{n+1}{2}} }{6(\ell +\gamma_0)(4\kappa^p(1+\ell+\gamma_0)+(1-\nu_\zeta^2)^{\frac{n+1}{2}})}
\quad   \textrm{
implique
$\dis
V:=\frac{12\kappa^p(1+k + \gamma_0)^2}{(1-6(k+\gamma_0)u)}\frac{(k+\gamma_0) u}{(1-\nu_\zeta^2)^{\frac{n+1}{2}}}\le \frac{1}{2}(1+k+\gamma_0)
$}\label{ineq2_gammak1}
\end{align}
En effet il est facile de voir que
pour $\dis u\le  \frac{(1-\nu_\zeta^2)^{\frac{n+1}{2}} }{6(k +\gamma_0)(4\kappa^p(1+k+\gamma_0)+(1-\nu_\zeta^2)^{\frac{n+1}{2}})}$
nous avons
$$\frac{2V}{1+k+\gamma_0}:=\frac{24\kappa^p(1+k + \gamma_0)}{(1-6(k+\gamma_0)u)}\frac{(k+\gamma_0) u}{(1-\nu_\zeta^2)^{\frac{n+1}{2}}}\le 1.$$
En tenant compte des in\'egalit\'es~(\ref{ineq1_gammak1}) et ~(\ref{ineq2_gammak1})
dans ~(\ref{gammak1})
il s'ensuit que $\gamma_{k+1}\le 1+k+\gamma_0$.
 \end{proof} 
\section{$\gamma$-th\'eor\`eme et $\alpha$-th\'eor\`eme pour un syst\`eme d\'eflat\'e}\label{sec_alpha_gamma}
Nous \'enon{\c{c}}ons un $\gamma-$th\'eor\`eme pour un syst\`eme d\'eflat\'e.
\begin{tm}\label{NS_gamma_tm}($\gamma$-th\'eor\`eme).
Soient $f \in \A2(\omega,R_{\omega})^s$ et $\zeta\in B(\omega, R_{\omega})$ une racine de $f$.
Soit $\ell$ l'\'epaisseur d'une suite de d\'eflation telle que pour tout $0\le k<\ell$, chaque \'el\'ement de la suite
 $F_0=S(f)$, $F_{k+1}=S(K(F_{k}))$,   satisfait $F_k(\zeta)=0$ et 
 $r_k:=\rank(DF_k(\zeta))~<~n$. Soient $p_k$ pour $k=0:\ell-1$ et $\dis p=\underset{k=0:\ell-1}{\max}p_k$.
Nous notons $\kappa=\kappa_\zeta$,
 $\mu =   \underset{k=0:\ell }{\max}||D_{1:r_k }F_{k,1:r_k}(\zeta)^{-1}||$,
 $\dis \gamma_0:=\frac{2p_0}{p_0+1} \lambda(f,\zeta)\,\kappa^{p_0-1}\,\mu$  et $\gamma_\ell:=\gamma_0+\ell$.
Soit  $R$ tel que
 $$\kappa\,R:=\min\left (\frac{1}{p+1},\quad
\frac{(1-\nu_\zeta^2)^{\frac{n+1}{2}} }{6\gamma_\ell\,\left (4\kappa^p(\gamma_\ell+1)+(1-\nu_\zeta^2)^{\frac{n+1}{2}}\right )},\quad\frac{2\gamma_\ell+1-\sqrt{ 4\gamma_\ell^2+3\gamma_\ell}}{\gamma_\ell+1}\, 		
 \right ).$$
Alors pour tout  $x\in B(\zeta, R)$ la suite de Newton d\'efinie par la table~\ref{Ndflx0},
$$x_0=x,\quad x_{k+1}=N_{\dfl(f)}(x_k),\quad k\ge 0,$$
   converge quadratiquement vers $\zeta$.
\end{tm}
\begin{proof}
Le th\'eor\`eme~\ref{estim_gammal} montre que $\gamma(F_\ell,\zeta)\le \gamma_\ell$ dans la boule $B(\zeta,R)$.
Nous appliquons alors le th\'eor\`eme~\ref{new_gamma_tm} au syst\`eme $F_\ell$ avec $\gamma_\ell$.
\end{proof}
Nous donnons \'egalement un r\'esultat d'existence d'une racine singuli\`ere
reposant sur le th\'eor\`eme~\ref{rouche_regular}.
  \begin{tm}\label{rouche_deflation}
  Soit $f \in \A2(\omega,R_{\omega})^s$ and $x_0\in B(\omega,R_{\omega})$. Supposons qu'il existe une suite de d\'eflation
  $\dis (F_k)_{0\le k \le \ell}$
d'\'epaisseur $\ell$ en $x_0$.
Plus pr\'ecis\'ement
\begin{itemize}
\item[1--] Pour  tout $0\le k<\ell$ chaque \'el\'ement $F_0=S(f)$, $F_{k+1}=S(K(F_{k}))$ satisfait 
\begin{itemize}
\item[1.1--] $\dis \|F_k(x_0)\|\le  \eta_k:=\frac{2\alpha_0}{(n+1)(n+2)(R_{x_0}+\|F_k\|)R_{x_0}^{n-2}}$;
\item[1.2--] $DF_k(x_0)$ poss\`ede un  $\varepsilon_k$-rang num\'erique strictement inf\'erieur \`a $n$ o\`u $\varepsilon_k$ est le $\varepsilon $ donn\'e en ligne
$5$ de la table 1.
\end{itemize}
\item[2--] Le syst\`eme $\dfl(f)$ satisfait les hypoth\`eses de l'$\alpha$-th\'eor\`eme~\ref{rouche_regular} en  $x_0$. 
\end{itemize} 
Alors     $f$ a une seule racine dans la boule  $B(x_0,\theta)$
o\`u $\theta$ est d\'efini dans l'$\alpha$-th\'eor\`eme~\ref{rouche_regular}.
  \end{tm}
\section{Exemple}
Donnons un exemple afin d'illustrer les algorithmes exact et
num\'erique, en consid\'erant
$f(x,y)=(f_1(x,y),f_2(x,y))$ avec
$$f_1(x,y)=x^3/3+y^2x+x^2+2yx+y^2,\quad 
f_2(x,y)=x^2y-y^2x+x^2+2yx+y^2.$$
Le z\'ero $(0,0)$ est de multiplicit\'e $6$.
\subsection{Calculs exacts}\label{exemple_exact}
Nous avons
$$ Df(x,y)=\left (\begin{array}{cc}
x^2+y^ 2+2x+2y&2xy+2x+2y\\2xy-y^2+2x+2y&x^2-2xy+2x+2y
\end{array}\right ).$$
Le rang en $(0,0)$ de la matrice jacobienne est $0$. Donc le
  premier cran de la suite de d\'eflation consiste juste \`a remplacer
  chaque \'equation du syst\`eme initial par son gradient~:
$$F_0:=S(f)=(x^2+y^2+2x+2y, 2xy+2x+2y, 2xy-y^2+2x+2y,
x^2-2xy+2x+2y).$$
Les 4 lignes de la matrice jacobienne de $F_0$ sont~:
$$\left (\begin{array}{cc}
2x+2&2y+2\\2y+2&2x+2\\
2y+2&2x-2y+2\\
2x-2y+2&-2x+2
\end{array}\right ).$$
Le rang en $(0,0)$ de la matrice $DF_0(0,0)=\left (\begin{array}{cc}
2&2\\2&2\\2&2 \\2&2
\end{array}\right )$ est $1$.\\
Le compl\'ement de Schur de $DF_0(x,y)$ associ\'e \`a la sous-matrice $2x+2$ est
$$
\schur(DF_1(x,y))=\frac{2}{x+1}\left (\begin{array}{c}
2x-2y+x^2-y^2
\\
2x-3y+x^2-xy-y^2
\\
-x-x^2-xy+y^2.
\end{array}\right )
$$
Nous pouvons facilement v\'erifier que le syst\`eme
$F_1=(f_1,vec(\schur(DF_0(x,y))$ est r\'egulier et \'equivalent en $(0,0)$
\`a $f$. Remarquons que le syst\`eme tronqu\'e de $F_1$ \`a l'ordre $1$ 
$$2(x+y,\, 2x-2y,\, 2x-3y,\, -x)$$
est un syst\`eme r\'egulier \'equivalent en $(0,0)$ \`a $f$.
  \subsection{Calculs num\'eriques}\label{calc_num}
Un code Maple reproduisant les calculs ci-dessous est téléchargeable
\`a \textcolor{blue}{\underline{\url{https://perso.math.univ-toulouse.fr/yak/curriculum-vitae/}}}.
Nous donnons le comportement de la suite de d\'eflation. 
    \begin{itemize}
       \item[1--] Le point initial
$(x_0,y_0)=(-0.0005, 0.0006)$.
    \item[2--] Le syst\`eme
\begin{align*}
f= \left (
  \begin{array}{ccc}
  1/3\,{x}^{3}+{y}^{2}x+{x}^{2}+2\,xy+{y}^{2}
\\
{x}^{2}y-y^2x+{x}^{2}+2\,xy+{y}^{2}
  \end{array}
       \right ).
\end{align*}
\item[3--] La boule $B(x_0,y_0, R_{x_0,y_0}):=B(x_0,y_0,1)$.
  \item[4--]
S\'erie tronqu\'ee  $Tr(f):=Tr_{(x_0,y_0),3}(f)$ du syst\`eme
$f(x+x_0,y+y_0)$.
\begin{align*}
Tr(f)=\left (
  \begin{array}{ccc}
  0.00000000978+ 0.000201\,x+ 0.000199\,y+ 1.0\,{x}^{2}+ 2.0\,xy+ 1.0
\,{y}^{2}+ 0.333\,{x}^{3}+{y}^{2}x
\\ 0.0000000103+ 0.000201\,y+
 0.000199\,x+ 2.0\,xy+ 1.0\,{x}^{2}+ 1.0\,{y}^{2}+{x}^{2}y- 1.0\,{y}^{
2}x
 \end{array}
       \right )
\end{align*}
     \item[5--] Calcul de $F_0=S(Tr(f)).$
    \begin{align*} F_0=\left (
  \begin{array}{ccc}
  0.00019940+ 2.0012\,x+ 1.9990\,y+ 2.0\,xy
  \\
   0.00019904+ 1.9978\,y+
 2.0012\,x+ 2.0\,xy- 1.0\,{y}^{2}
 \\
  0.00020061+ 1.9990\,x+ 2.0012\,y+
 1.0\,{x}^{2}+{y}^{2}
 \\
  0.00020085+ 1.9978\,x+ 2.0010\,y+{x}^{2}- 2.0\,
xy
 \end{array}
       \right )
       \end{align*}
       Le tableau ci-dessous donne le d\'etail des tests effectu\'es par l'algorithme de s\'election qui conduisent \`a ce syst\`eme.
       $$
\hspace{-2cm} \begin{array}{|c|c|c|c|}
\hline
&&&\\
\textrm{Fonction} & \textrm{\'Evaluation en $(0,0)$}&\eta&F_0
\\\hline
\scr Tr(f)_1 & 10^{-8}   &0.014& 
\\
\hline
\scr\partial_y Tr(f)_1&2\times 10^{-4} &0.01& F_{01}
\\
\hline
\scr\partial^2_{yy} Tr(f)_1&1.99 &0.0078& 
\\
\hline
\scr\partial^2_{xy} Tr(f)_1&2.012 &0.0078& 
\\
\hline
\scr\partial_x Tr(f)_1&2\times 10^{-4} &0.0098& F_{02}
\\
\hline
\scr
Tr(f)_2 & 10^{-8}   &0.014& 
\\
\hline
\scr\partial_x Tr(f)_2&2\times 10^{-4} &0.0098& F_{03}
\\
\hline
\scr\partial^2_{xx} Tr(f)_2&2.01&0.0078& 
\\
\hline
\scr\partial^2_{xy} Tr(f)_2&1.99 &0.0074& 
\\
\hline
\scr\partial_x Tr(f)_2&2\times 10^{-4} &0.0098& F_{04}
\\
\hline
  \end{array}
  $$
     \item[6--] Nous avons successivement
     $\|F_0\|=2.4045$,
        \\$  \dis \eta=\frac{2\alpha_0}{12(R_{x_0}+\|F_0\|)R_{x_0}^{n-2}}=0.0064 ~> 
~\|F_0(0,0)\|~=~2.8\times 10^{-4}	.$
       \item[7--] Jacobienne de $F_0$ en $(0,0)$:
       $
   DF_0(0,0)=\left( \begin {array}{cc}  2.00012& 1.999\\ 
\noalign{\medskip}
 2.00012& 1.9978\\ 
\noalign{\medskip}
1.999& 2.00012\\ 
\noalign{\medskip}
1.9978& 2.0001
 \end {array} \right )
     $.
     Les valeurs singuli\`eres de cette jacobienne sont $0.0039$ et $5.6562$.
     Cette jacobienne a un $\varepsilon_0=0.008$-rang \'egal \`a $1$.%

    \item[8--] D\'enoyautage de $F_0$ en $(0,0)$ :
\begin{align*}
K(F_0)=\left (\begin{array}{ccc}
   0.00019940+ 2.0012\,x+ 1.9990\,y
   \\
   - 0.0012- 2.0\,y
   \\
    0.0043976+ 3.9956
\,y- 3.9956\,x
\\ 0.0053963- 5.9944\,x+ 3.9922\,y
\end{array}
\right ).
\end{align*}
    \item[9--] $F_1:=S(K(F_0))=K(F_0)$.
     \item[9--] \'Evaluation de $F_1$ en $(0,0)$ : $F_1(0,0)=(-0.12e-2, 0.19940e-3, 0.43976e-2, 0.53963e-2)$.
     Nous avons $\|F_1\|=3.9048$ et
      \\$\dis \eta=\frac{2\alpha_0}{12(R_{x_0}+\|F_1\|)R_{x_0}^{n-2}}=0.0044> 
\|F_1(0,0)\|=0.0012165.$
        \item[10--] Matrice jacobienne de $F_1$ et son \'evaluation en
$(0,0)$:
       $$  DF_1(x,y)=\left( \begin {array}{cc} 0&-2\\ 
\noalign{\medskip} 2.0012& 1.999
\\\noalign{\medskip}
  -3.9956& 3.9956
  \\\noalign{\medskip}
  -5.9944&3.9922
  \end {array} \right )
     $$
     Les valeurs singuli\`eres de $DF_1(0,0)$ sont $9.2$ et $3.34$ et son 
$\varepsilon_1=3.34$-rang est $2$.	
\item[11--] Extraction d'un syst\`eme r\'egulier de $F_1$ en $(0,0)$. 
     \begin{align*}
  dfl(f)=
  \left(\begin{array}{cccc}
  0.00019940+ 2.0012\,x+ 1.9990\,y
  \\
   0.0043976+ 3.9956\,y- 3.9956\,x
\end{array}
  \right )
\end{align*}
Nous trouvons que l'it\'er\'e de $(x_0,y_0)=(-0.0005,0.0006)$ par l'op\'erateur
de Newton est $(1.5\times 10^{-7},-4.5\times 10^{-7})$. Ceci illustre la 
propri\'et\'e
d'une convergence quadratique au premier pas de l'it\'eration.
   \end{itemize}
  Celle-ci est confirm\'ee par le comportement des it\'er\'es
  succesifs d\'etermin\'es par l'algorithme \textsf{Newton singulier}.
\begin{align*}
&[- 0.0005, 0.0006]
\\\\
&
[-1.5\times 10^{-7},-4.5\times 10^{-7}]
\\\\
&
[ 10^{-13 },  -1.7\times 10^{-13 }]
\\\\
&
[{ 9.6\times 10^{-27}},{ -2.6\times 10^{-26}}]
\\\\
&
[{ 3.5\times 10^{-52}},{ -6.13\times 10^{-52}}]
\\\\
&
[{ 1.2\times 10^{-103}},{ -3.4\times 10^{-103}}]
\\\\
&
[{ 5.9\times 10^{-206}},{ -1.02\times 10^{-206}}]
\end{align*}
\subsection{Illustration des th\'eor\`emes
~\ref{rouche_regular} et~\ref{new_gamma_tm}. }
Donn\'ee par la table ci-dessous~:
  $$
\hspace{-2cm} \begin{array}{|c|c|c|c|c|c|c|c|}
\hline
&&&&&&&\\
&\beta&\kappa&\gamma&\alpha&\scr 
2\gamma+1-\sqrt{(2\gamma+1)^2-1}&\theta=
\frac{\alpha+1-\sqrt{(\alpha+1)^2-4\alpha(\gamma+1)}}{2\kappa(\gamma+1)}&
\frac{2\gamma+1-\sqrt{ 4\gamma^2+3\gamma}}{\kappa(\gamma+1)}
\\\hline
(x_0,y_0) & 0.00078  &3 &2.68&0.00234&0.079&0.000786&
\\
\hline
\zeta=(0,0)& 0&3&2.67& & & & 	0.0269
\\
\hline
  \end{array}
  $$
\vfill\eject\null

\appendix\section{Feuille de calcul Maple de la sous-section~\ref{calc_num}}
{\textbf{Restart}}
{\advance\baselineskip -1pt 
\begin{verbatim}
restart:
with(LinearAlgebra):
print_f:=proc(f)
local k;
for k to nops(f) do print(f[k]);od;
end:
\end{verbatim}
}	
{\textbf{Constantes et proc\'edure de la norme L2}}
\vspace{-0.1cm}

{\advance\baselineskip -1pt 
\begin{verbatim}
(1-4*u+2*u^2)^2-2*u;
alpha0:=fsolve(%,u=0..1);
c0:=evalf(sum((1/2)^(2^k-1),k=0..infinity),20);
\end{verbatim}
}
\vspace{-0.8cm}

\begin{align*}
\left( 2\,{u}^{2}-4\,u+1 \right) ^{2}-2\,u
 \\\alpha_0:=0.1307169444
\\ c_0:=1.6328430180437862874
\end{align*}
{\advance\baselineskip -1pt 
\begin{verbatim}
Norm_L2:=proc(f,x0,r)
local i,N,c,k; global n,x,y;
N:=0;  c:=2/Pi^n/r^(2*n);
for k to nops(f) do
   N:=N+evalf(int(f[k]^2,[y=x0[2]-sqrt(r^2-(x-x0[1])^2)..
   x0[2]+sqrt(r^2-(x-x0[1])^2),x=x0[1]-r..x0[1]+r]),20);
od:
evalf(sqrt(c*N));
end:
\end{verbatim}
}
{\textbf{Proc\'edure S de s\'election}}
{\advance\baselineskip -1pt 
\begin{verbatim}
S:=proc(f,x0,r)
local eta,i; global Sf,Sf1,alpha0; global x,y;
for i to nops(f) do
  eta:=2*alpha0/12/(r+Norm_L2({f[i]},x0,r))/r;
  if eta>=evalf(abs(subs(x=x0[1],y=x0[2],f[i]))) then
     Sf1:=f[i]; 
     S({diff(f[i],x),diff(f[i],y)} minus {0},x0,r);
  else
     Sf:={op(Sf),Sf1};
  fi;
od; 
end:
\end{verbatim}
}
{\textbf{Proc\'edure S de s\'election d\'etaill\'ee}}
{\advance\baselineskip -1pt 
\begin{verbatim}
S_print:=proc(f,x0,r)
local eta_,ei,si,i,j,nopsSf; global Sf,Sf1,alpha0; global x,y,n;
for i to nops(f) do
  print();
  ei:=[seq(f[i][j],j=2..3)]; si:=ei[1]+ei[2];
  if si=0 then n:=i;fi;
  eta_:=2*alpha0/12/(r+Norm_L2({f[i][1]},x0,r))/r;
  if si=0  then printf("%s%g%s%v",`f[`,n,`]=`,evalf(f[i][1],5));
  else printf("%s%g%s%g%s%g%s%v",`dérivée de f[`,n,`] à  l'ordre
       [`,ei[1],`,`,ei[2],`]=`,evalf(f[i][1],5)); fi;
  if eta_>=evalf(abs(subs(x=x0[1],y=x0[2],f[i][1]))) then
     print();print(`évaluation en (0,0) =`,evalf(abs(subs(x=x0[1],y=x0[2],f[i][1])),5),
     `<  eta=`,evalf(eta_,5));
     Sf1:=f[i][1]; 
     S_print({[diff(Sf1,x),ei[1]+1,ei[2]],[diff(Sf1,y),ei[1],ei[2]+1]}        
              minus {[0,ei[1]+1,ei[2]],[0,ei[1],ei[2]+1]},x0,r);
  else
     print();print(`evaluation en (0,0) =`,evalf(abs(subs(x=x0[1],y=x0[2],f[i][1])),5),
          `>  eta=`,evalf(eta_,5));
     nopsSf:=nops(Sf);
     Sf:={op(Sf),Sf1};
     if nops(Sf)>nopsSf then  print();print(`on retient la fonction `,evalf(Sf1,5));   
     else print();print(`la fonction `,evalf(Sf1,5), `  est déjà  retenue`);
     fi:
  fi;
od;
end:
\end{verbatim}
}

{\textbf{Proc\'edure de d\'etermination du rang num\'erique}}
{\advance\baselineskip -1pt 
\begin{verbatim}
P:=proc(s)
local i; global lambda;
expand(product(lambda-s[i],i=1..nops(s)));
end:

numerical_rank:=proc(A)
local s,p,bg,k,B,i,m,r,epsilon,e,n,beta,gama,ar,alpha0,gr; global lambda;
m,n:=Dimension(A);
B:=A;
if n<m then B:=Transpose(A);fi;
s:=SingularValues(B,output='list');
p:=P(s);
print(p);
bg:=[]:
n:=nops(s):
for k to n do
   if coeff(p,lambda,n-k)<>0 then
      e:=seq((abs(coeff(p,lambda,i)/coeff(p,lambda,k)))(1/(k-i)),i=0..k-1);
      beta:=max(%);
      gama:=1;
      if k< n then
        seq((abs(coeff(p,lambda,i)/coeff(p,lambda,k)))^(1/(i-k)),i=k+1..n);
        gama:=max(%);
      fi;
   else beta:=1;gama:=1; fi;
   bg:=[op(bg),[beta,gama]];
od;
print(bg);
r:=0;  
for k to nops(bg) do
  if bg[k][1]*bg[k][2]<= 1.0/9 then  
      ar:=bg[k][1]*bg[k][2];  gr:=bg[k][2];   r:=k;
  fi;
od; 
if r=0 then epsilon:=s[n]; 
else
   epsilon:=(3*ar+1-sqrt((3*ar+1)^2-16*ar))/4/gr;
fi;
print(valeurs_singulières,evalf(s,5));
print(epsilon_,evalf(epsilon,5));
print(rang_numérique,n-r);
[n-r,epsilon,s];
end:
\end{verbatim}
}
{\textbf{Point initial, syst\`eme, rayon}}
{\advance\baselineskip -1pt 
\begin{verbatim}
Digits:=250:
x0y0:=[-0.0005,0.0006];
n:=2:
f:=[x^3/3+y^2*x+x^2+2*x*y+y^2, x^2*y-x*y^2+x^2+2*x*y+y^2]:
print_f(%);
Rx0y0:=1.0;
\end{verbatim}
}
\vspace{-0.75cm}

\begin{align*}
x0y0:=[- 0.0005, 0.0006]
\\
f:=[1/3\,{x}^{3}+{y}^{2}x+{x}^{2}+2\,xy+{y}^{2},{x}^{2}y-{y}^{2}x+{x}^{2}
+2\,xy+{y}^{2}]
\\
 Rx0y0:=1.0
\end{align*}
{\textbf{S\'erie tronqu\'ee \`a l'ordre 3 de f(x+x0,y+y0)}}
{\advance\baselineskip -1pt 
\begin{verbatim}
x0:=x0y0[1]:  y0:=x0y0[2]:
subs(x=x+x0,y=y+y0,f): 
f:=[seq(mtaylor(%[i],[x,y],4),i in {1,2})]:
print_f(evalf(%,5));
\end{verbatim}
}
\vspace{-0.75cm}

\begin{align*}
0.00000000978+ 0.000201\,x+ 0.000199\,y+ 1.0\,{x}^{2}+ 2.0\,xy+ 1.0
\,{y}^{2}+ 0.333\,{x}^{3}+{y}^{2}x
\\
 0.0000000103+ 0.000201\,y+
 0.000199\,x+ 2.0\,xy+ 1.0\,{x}^{2}+ 1.0\,{y}^{2}+{x}^{2}y- 1.0\,{y}^{
2}x
\end{align*}
{\textbf{Calcul de F0:=S(f)}}
{\advance\baselineskip -1pt 
\begin{verbatim}
Sf:={}:
S({op(f)},[0.0,0.0],1.0):
F0:=[op(Sf)]:
print_f(evalf(%,5));
\end{verbatim}
}
\begin{align*}
0.00019940+ 2.0012\,x+ 1.9990\,y+ 2.0\,xy
\\
 0.00019904+ 1.9978\,y+
 2.0012\,x+ 2.0\,xy- 1.0\,{y}^{2}
 \\ 0.00020061+ 1.9990\,x+ 2.0012\,y+
 1.0\,{x}^{2}+{y}^{2}
 \\ 0.00020085+ 1.9978\,x+ 2.0010\,y+{x}^{2}- 2.0\,
xy
\end{align*}
{\textbf{D\'etail des calculs de F0:=S(f)}}
{\advance\baselineskip -1pt 
\begin{verbatim}
Sf:={}:
g:=[seq([f[i],0,0],i=1..nops(f))]:
S_print({op(g)},[0.0,0.0],1.0):
\end{verbatim}
}
{\advance\baselineskip -6pt 
\begin{verbatim}

f[1]=.97783e-8+.20061e-3*x+.19940e-3*y+.99950*x^2+2.0012*x*y+.9995*y^2+.33333*x^3+y^2*x

     évaluation en (0,0) = 9.7783 10^{-9}   <  eta= 0.010711

dérivée de f[1] à  l'ordre [0,1]=.19940e-3+2.0012*x+1.9990*y+2.*x*y

     évaluation en (0,0) = 0.00019940 <  eta= 0.0070694

dérivée de f[1] à  l'ordre [0,2]= 1.9990+2.*x 

     évaluation en (0,0) = 1.9990 >  eta= 0.0052358

on retient la fonction , 0.00019940 + 2.0012 x + 1.9990 y + 2. x y

dérivée de f[1] à  l'ordre [1,1]= 2.0012+2.*y 

     évaluation en (0,0) = 2.0012 >  eta= 0.0052323

la fonction  0.00019940 + 2.0012 x + 1.9990 y + 2. x y  est déjà  retenue

dérivée de f[1] à  l'ordre [1,0]= .20061e-3+1.9990*x+2.0012*y+1.0000*x^2+y^2 

     évaluation en (0,0) = 0.00020061 <  eta= 0.0068934

dérivée de f[1] à  l'ordre [2,0]= 1.9990+2.0000*x 

     évaluation en (0,0) = 1.9990 >  eta= 0.0052358

on retient la fonction 0.00020061 + 1.9990 x + 2.0012 y + 1.0000 x^2  + y^2 

dérivée de f[1] à  l'ordre [1,1]= 2.0012+2.*y 

     évaluation en (0,0) = 2.0012 >  eta= 0.0052323
        
la fonction  0.00020061 + 1.9990 x + 2.0012 y + 1.0000 x^2  + y^2   est déjà  retenue

f[2]= .10330e-7+.20085e-3*y+.19904e-3*x+1.9978*x*y+1.0006*x^2+1.0005*y^2+x^2*y-1.*y^2*x 
                                                       
     évaluation en (0,0) = 1.0330 10^{-8}   <  eta= 0.013775

dérivée de f[2] à  l'ordre [1,0]= .19904e-3+1.9978*y+2.0012*x+2.*x*y-1.*y^2 

     évaluation en (0,0) = 0.00019904 <  eta= 0.0098688

dérivée de f[2] à  l'ordre [2,0]= 2.0012+2.*y 

     évaluation en (0,0) = 2.0012 >  eta= 0.0078227

on retient la fonction 0.00019904 + 1.9978 y + 2.0012 x + 2. x y - 1. y^2 

dérivée de f[2] à  l'ordre [1,1]= 1.9978+2.*x-2.*y 

     évaluation en (0,0) = 1.9978 >  eta= 0.0073777
                                                                
la fonction  0.00019904 + 1.9978 y + 2.0012 x + 2. x y - 1. y^2   est déjà  retenue

dérivée de f[2] à  l'ordre [0,1]= .20085e-3+1.9978*x+2.0010*y+x^2-2.*x*y 

     évaluation en (0,0) = 0.00020085 <  eta= 0.0098688

dérivée de f[2] à  l'ordre [0,2]= 2.0010-2.*x 

     évaluation en (0,0) = 2.0010 >  eta= 0.0078231

on retient la fonction  0.00020085 + 1.9978 x + 2.0010 y + x^2  - 2. x y

dérivée de f[2] à  l'ordre [1,1]= 1.9978+2.*x-2.*y 

     évaluation en (0,0) = 1.9978 >  eta= 0.0073777
         
la fonction  0.00020085 + 1.9978 x + 2.0010 y + x^2  - 2. x y est déjà  retenue

\end{verbatim}
}
{\textbf{\'Evaluation en (0,0) et norme L2 de F0}}
{\advance\baselineskip -1pt 
\begin{verbatim}
eval_F0:=subs(x=0,y=0,F0): print(evalf(%,5));
NF0:=Norm(Vector(2,%),2):  print(evalf(%,5));
NL2F0:=Norm_L2(F0,[0,0],Rx0y0): print(evalf(%,5));
\end{verbatim}
}
\vspace{-0.75cm}

\begin{align*}
        [0.00019940, 0.00019904, 0.00020061, 0.00020085]
        \\
                           0.00028174
                           \\
                             2.4045
\end{align*}
{\textbf{Test $||$F0(0,0)$||<$eta}}
{\advance\baselineskip -1pt 
\begin{verbatim}
eta:=2*alpha0/12/(Rx0y0+NL2F0)/Rx0y0^(n-2):
NF0<eta:  print(evalf(%,5));
\end{verbatim}
}
\vspace{-0.75cm}

\begin{align*}
        0.28174e-3 < 0.63992e-2
\end{align*}
{\textbf{Valeurs singuli\`eres et rang num\'erique de la jacobienne de F0 en (0,0)}}
{\advance\baselineskip -1pt 
\begin{verbatim}
J:=VectorCalculus[Jacobian](F0,[x,y]): print(evalf(%,5));
J0:=subs(x=0,y=0,J):  print(evalf(%,5));
numerical_rank(J0):
\end{verbatim}
}
\vspace{-0.75cm}

\begin{align*}
      \left[ \begin {array}{cc}  2.0012+2\,y& 1.9990+2\,x
\\ \noalign{\medskip} 2.0012+2\,y& 1.9978+2\,x- 2.0\,y
\\ \noalign{\medskip}
 1.999
+ 2.0\,x& 2.0012+2\,y\\ \noalign{\medskip} 1.9978+2\,x- 2.0\,y& 2.0010
- 2.0\,x\end {array} \right]
\\
\left[ \begin {array}{cc}  2.0012& 1.9990\\ \noalign{\medskip} 2.0012
& 1.9978\\ \noalign{\medskip}
 1.999
& 2.0012\\ \noalign{\medskip} 1.9978& 2.0010\end {array} \right] 
\\
            \textrm{valeurs\,singuli\`eres}= [5.6562, 0.0039667]\\
                      \epsilon= 0.0079335\\
                      \textrm{ rang num\'erique}= 1
\end{align*}
{\textbf{D\'enoyautage de F0 en (0,0)}}
{\advance\baselineskip -1pt 
\begin{verbatim}
J(2..4,2)-J(2..4,1)*J(1,2)/J(1,1):
F1:=map(mtaylor,[F0[1],%[1],%[2],%[3]],[x,y],2):
print_f(evalf(%,5));
\end{verbatim}
}
\vspace{-0.75cm}

\begin{align*}
      0.00019940+ 2.0012\,x+ 1.9990\,y
      \\
      - 0.0012- 2.0\,y
      \\
       0.0043976+ 3.9956
\,y- 3.9956\,x
\\
 0.0053963- 5.9944\,x+ 3.9922\,y
\end{align*}
{\textbf{Calcul de F1:=S(F1)}}
{\advance\baselineskip -1pt 
\begin{verbatim}
Sf:={}:
S({op(F1)},[0.0,0.0],1.0):
F1:=[op(Sf)]:
print_f(evalf(%,5));
\end{verbatim}
}
\vspace{-0.75cm}

\begin{align*}
       - 0.0012- 2.0\,y
       \\
        0.00019940+ 2.0012\,x+ 1.9990\,y
        \\
         0.0043976+ 3.9956
\,y- 3.9956\,x
\\ 0.0053963- 5.9944\,x+ 3.9922\,y
\end{align*}
{\textbf{\'Evaluation en (0,0) et norme L2 de F1}}
{\advance\baselineskip -1pt 
\begin{verbatim}
eval_F1:=subs(x=0,y=0,F1): print(evalf(%,5));
NF1:=Norm(Vector(2,%),2):  print(evalf(%,5));
NL2F1:=Norm_L2(F1,[0,0],Rx0y0): print(evalf(%,5));
\end{verbatim}
}
\vspace{-0.75cm}

\begin{align*}
                 [-0.0012, 0.00019940, 0.0043976, 0.0053963]\\
                           0.0012165\\
                             3.9048\\
\end{align*}
{\textbf{Test $||$F1(0,0)$||<$eta}}
{\advance\baselineskip -1pt 
\begin{verbatim}
eta:=2*alpha0/12/(Rx0y0+NL2F1)/Rx0y0^(n-2):
NF1<eta:  print(evalf(%,5));
\end{verbatim}
}
\vspace{-0.75cm}

\begin{align*}
 0.12165e-2 < 0.44418e-2
\end{align*}
{\textbf{Valeurs singuli\`eres et rang num\'erique de la jacobienne de F1 en (0,0)}}
{\advance\baselineskip -1pt 
\begin{verbatim}
eta:=2*alpha0/12/(Rx0y0+NL2F0)/Rx0y0^(n-2):
NF0<eta:  print(evalf(%,5));
\end{verbatim}
}
\vspace{-0.75cm}

\begin{align*}
        \left[ \begin {array}{cc}  0.0&- 2.0\\ \noalign{\medskip} 2.0012&
 1.9990\\ \noalign{\medskip}- 3.9956& 3.9956\\ \noalign{\medskip}-
 5.9944& 3.9922\end {array} \right]
 \\
 \textrm{valeurs\,singuli\`eres}=[9.2020, 3.3353]\\
                      \epsilon= 3.3353\\
                      \textrm{ rang num\'erique}= 2
\end{align*}
{\textbf{Extraction d'un syst\`eme r\'egulier}}
{\advance\baselineskip -1pt 
\begin{verbatim}
R:=[seq(F1[i],i={2,3})]: print_f(evalf(%,5));
\end{verbatim}
}
\vspace{-0.75cm}

\begin{align*}
       0.00019940+ 2.0012\,x+ 1.9990\,y
       \\
        0.0043976+ 3.9956\,y- 3.9956\,x
\end{align*}
{\textbf{It\'er\'e de (x0,y0) par l'op\'erateur de Newton singulier}}
{\advance\baselineskip -1pt 
\begin{verbatim}
eta:=2*alpha0/12/(Rx0y0+NL2F0)/Rx0y0^(n-2):
NF0<eta:  print(evalf(%,5));
\end{verbatim}
}
\vspace{-0.75cm}

\begin{align*}
        0.28174e-3 < 0.63992e-2
\end{align*}
{\textbf{D\'eflation et it\'erations successives}}
{\advance\baselineskip -1pt 
\begin{verbatim}
deflation:=proc(X0)
local f, X, x0, y0, J, F0,F1,i; global Sf,x,y;
f:=[x^3/3+y^2*x+x^2+2*x*y+y^2, x^2*y-x*y^2+x^2+2*x*y+y^2];
x0:=X0[1]:  y0:=X0[2]:
subs(x=x+x0,y=y+y0,f): #print(f);
F0:=[seq(mtaylor(%[i],[x,y],4),i in {1,2})]; #print(%);
Sf:={}:
S({op(F0)},[0.0,0.0],1.0):
F0:=[op(Sf)]:
J:=VectorCalculus[Jacobian](F0,[x,y]): 
J(2..4,2)-J(2..4,1)*J(1,2)/J(1,1): 
F1:=map(mtaylor,[F0[1],%[2]],[x,y],2); 
op(solve(F1,[x,y])):
[rhs(%[1])+x0,rhs(%[2])+y0];
end:

Digits:=250:
X0:=[-0.0005,0.0006];
for i to 6 do 
  X0:=deflation(X0): print(evalf(X0,5));
od: 
\end{verbatim}
}
\vspace{-0.75cm}

\begin{align*}
                              [-0.0005, 0.0006]
\\
[ 0.00000015231,- 0.00000045263]
\\
[{ 1.0038\times 10^{-13}},-{ 1.6932\times 10^{-13}}]
\\
[{ 9.6859\times 10^{-27}},-{ 2.6681\times 10^{-26}}]
\\
[{ 3.5521\times 10^{-52}},-{ 6.1365\times 10^{-52}}]
\\
[{ 1.2568\times 10^{-103}},-{ 3.4366\times 10^{-103}}]
\\
[{ 5.9038\times 10^{-206}},-{ 1.0223\times 10^{-205}}]
\end{align*}
{\textbf{Th\'eor\`eme-alpha 9}}
{\advance\baselineskip -1pt 
\begin{verbatim}
alpha_theorem:=proc(f,x0,rho)
local Nf,kappa,J,beta,gama,alpha,theta,nu;  global n,x,y;
nu:=Norm(Vector(2,x0))/rho;
kappa:=max(1,(n+1)/rho/(1-nu^2));
VectorCalculus[Jacobian](f,[x,y]);
J:=MatrixInverse(subs(x=0,y=0,%));
%.Vector(2,subs(x=0,y=0,f));
beta:=Norm(%,2);
Nf:=Norm_L2(f,[0,0],rho); 
gama:=max(1,Nf*kappa*Norm(J,2)/(1-nu^2)^((n+1)/2)); 
alpha:=beta*kappa;
theta:=0; 
if alpha < 2*gama+1-sqrt((2*gama+1)^2-1) then
   theta:=(alpha+1-sqrt((alpha+1)^2-4*alpha*(gama+1)))/2/kappa/(gama+1);
fi;  
print(beta_,evalf(beta,5));
print(gamma_,evalf(gama,5));
print(kappa_,evalf(kappa,5));
alpha<2*gama+1-sqrt((2*gama+1)^2-1); print(test_alpha,evalf(%,5));
print(theta_,evalf(theta,5));
end:

Digits:=50:
X0:=[-0.0005,0.0006];  print(evalf(F1,5));
alpha_theorem([seq(F1[i],i in {2,3})],X0,Rx0y0):

\end{verbatim}
}
\vspace{-1.25cm}

\begin{align*}
                              \beta= 0.00078147
                         \\\gamma= 2.6737
                         \\\kappa= 3.0000
                \\\textrm{testalpha : } 0.0023444 < 0.079267
                    \\   \theta= 0.00078644
                    \end{align*}
                    \vspace{-0.75cm}

{\textbf{Th\'eor\`eme-gamma 11}}
{\advance\baselineskip -1pt 
\begin{verbatim}
gamma_theorem:=proc(f,x0,rho,r)
local Nf,NormJ,kappa,J,beta,gama,alpha,nu;  global  x,y;
nu:=Norm(Vector(2,x0))/rho;
kappa:=max(1,(n+1)/rho/(1-nu^2));
VectorCalculus[Jacobian](f,[x,y]);
if r>0 then
  J:=MatrixInverse(subs(x=0,y=0,%[1..r,1..r]));
  NormJ:=Norm(J,2);
elseyyygggxc
fi; 
Nf:=Norm_L2(f,x0,rho); 
gama:=max(1,Nf*kappa*NormJ/(1-nu^2)^((n+1)/2));
print(kappa_,evalf(kappa,5));
print(gamma_,evalf(gama,5));
(2*gama+1-sqrt(4*gama^2+3*gama))/kappa/(gama+1);
print(rayon,evalf(%,5));
end:

gamma_theorem([2*(x+y),4*(x-y)], [0,0],1,2);
\end{verbatim}
}
\vspace{-0.75cm}

\begin{align*}
                           \kappa= 3.
                         \\\gamma= 2.6761
                        \\rayon= 0.026858
\end{align*}


\bibliographystyle{acm}

 \end{document}